\documentclass[12pt,a4paper]{article}
\usepackage[margin=1in]{geometry}
\usepackage[T1]{fontenc}
\usepackage[utf8]{inputenc}
\usepackage{lmodern}
\usepackage{amsmath,amssymb,amsthm}
\usepackage{mathtools}  \usepackage{enumitem}

   \usepackage[numbers,sort&compress]{natbib}   \usepackage[colorlinks=true,linkcolor=blue,citecolor=blue,urlcolor=blue]{hyperref}
\hypersetup{pdftitle={Well-posedness of fully coupled McKean-Vlasov FBSDEs with jumps under full-tuple law dependence},pdfauthor={Chunrong Feng, Christian Garry},
            pdfkeywords={McKean-Vlasov FBSDEs, full-tuple law dependence, infinite jump activity, expected diagonal G-monotonicity, monotone continuation, Wasserstein distance}}
\newcommand{\R}{\mathbb{R}}
\newcommand{\E}{\mathbb{E}}
\newcommand{\Prob}{\mathbb{P}}
\newcommand{\F}{\mathcal{F}}
\newcommand{\diff}{\mathrm{d}}
\newcommand{\Nt}{\widetilde{N}}
\newcommand{\inner}[2]{\langle #1,\,#2\rangle}
\newcommand{\Lnu}{L^2(\nu)}
\newcommand{\Gstar}{G^{\!\top}}
\newcommand{\Law}{\mathrm{Law}}
\newif\ifshowdoi\showdoitrue

\AtBeginDocument{\let\doi\PaperDoi}

\newcommand{\HH}{\mathbb{S}}
\newcommand{\Ssp}[1]{\mathcal{S}^2(0,T;#1)}
\newcommand{\HW}{\mathcal{H}^2_W}
\newcommand{\HN}{\mathcal{H}^2_N}
\newcommand{\Lsp}[1]{L^2(0,T;#1)}
\newcommand{\Hweak}{\mathbb{H}^2}
\newcommand{\Lnorm}[1]{\lVert #1\rVert_{\Lnu}}

\providecommand{\PaperRunIn}[1]{\paragraph{#1.}}

\providecommand{\PaperPageBreak}{}

\providecommand{\PaperSystemBox}[1]{\mathmakebox[0.9\linewidth][c]{#1}}

\newcommand{\PaperEmailFeng}{chunrong.feng@durham.ac.uk}
\newcommand{\PaperEmailGarry}{christian.t.garry@durham.ac.uk}
\newcommand{\PaperDept}{Department of Mathematical Sciences, Durham University}
\newcommand{\PaperStreet}{South Road}
\newcommand{\PaperCity}{Durham}
\newcommand{\PaperPostcode}{DH1 3LE}
\newcommand{\PaperCountry}{United Kingdom}
 \numberwithin{equation}{section}
\theoremstyle{plain}
\newtheorem{theorem}{Theorem}[section]
\newtheorem{proposition}[theorem]{Proposition}
\newtheorem{lemma}[theorem]{Lemma}
\newtheorem{corollary}[theorem]{Corollary}
\theoremstyle{definition}
\newtheorem{assumption}[theorem]{Assumption}
\newtheorem{definition}[theorem]{Definition}
\providecommand{\PaperRemarkStyle}{remark}
\theoremstyle{\PaperRemarkStyle}
\newtheorem{remark}[theorem]{Remark}
\newtheorem{example}[theorem]{Example}
 \title{Well-posedness of fully coupled McKean--Vlasov\\
FBSDEs with jumps under full-tuple law dependence}
\author{Chunrong Feng and Christian Garry$^{*}$\\[4pt]
\small \PaperDept\\[1pt]
\small \PaperStreet, \PaperCity\ \PaperPostcode, \PaperCountry\\[3pt]
\small \texttt{\PaperEmailFeng}\\[1pt]
\small $^{*}$Corresponding author: \texttt{\PaperEmailGarry}}
\date{21 August 2026}
\begin{document}
\maketitle
\begin{abstract}
We prove existence, uniqueness, and stability for fully coupled McKean--Vlasov
forward--backward SDEs with jumps whose drift, diffusion, jump, and driver coefficients
may depend Lipschitz-continuously, in quadratic Wasserstein distance, on the joint law of
the full solution tuple $\Theta=(X,Y,Z,U)$: forward state, backward variable, Brownian
integrand, and $\Lnu$-valued jump integrand. The terminal function may depend
Lipschitz-continuously on $X_T$ and its law. The system is driven by a Brownian motion and an independent compensated Poisson
random measure with arbitrary $\sigma$-finite intensity, so infinite jump activity is
admitted.
Both the Lipschitz and monotonicity hypotheses are imposed only along diagonal
tuple--law pairs $(\Theta,\Law(\Theta))$; we show that expected diagonal monotonicity is
strictly weaker than pointwise monotonicity.

Under a jump-extended $G$-monotonicity condition we
establish an a priori continuous-dependence estimate, uniqueness, and existence on every
prescribed finite horizon, by monotone continuation in the coupling strength from a
small-coupling base case. A mean-field dealer-market example realises the $U$-law
dependence non-perturbatively: its law interaction is monotone at every interaction
strength, and its mark measure has infinite activity.
\end{abstract}



\section{Introduction}

We prove existence, uniqueness, and stability, on every fixed finite horizon $T$ and
with no smallness restriction, for the fully coupled
McKean--Vlasov forward--backward system with jumps
\begin{equation}\label{eq:mvfbsdej}
\PaperSystemBox{\left\{\;\begin{aligned}
X_t&=\xi_0+\int_0^t b(s,\Theta_s,\mu_s)\,\diff s+\int_0^t\sigma(s,\Theta_s,\mu_s)\,\diff W_s
  +\int_0^t\!\!\int_E h(s,\Theta_{s}^-,e,\mu_s^-)\,\Nt(\diff s,\diff e),\\[9pt]
Y_t&=g(X_T,\Law(X_T))+\int_t^T f(s,\Theta_s,\mu_s)\,\diff s
  -\int_t^T Z_s\,\diff W_s-\int_t^T\!\!\int_E U_s(e)\,\Nt(\diff s,\diff e),
\end{aligned}\right.}
\end{equation}
driven by a Brownian motion $W$ and an independent compensated Poisson random measure
$\Nt$. The intensity $\nu$ is assumed only $\sigma$-finite: infinite jump
activity, $\nu(E)=\infty$, is admitted. The state space is the
Hilbert space
\[
\mathcal H:=\R^n\times\R^m\times\R^{m\times d}\times L^2(\nu;\R^m),
\]
carrying the norm
\[
|\theta|_{\mathcal H}^2:=|x|^2+|y|^2+|z|_F^2+\Lnorm{u}^2,\qquad \theta=(x,y,z,u),
\]
where $|\cdot|$ is the Euclidean norm (on any $\R^j$) and the Frobenius and
$L^2(\nu;\R^m)$ norms are
\[
|z|_F^2:=\sum_{i,j}z_{ij}^2,
\qquad
\Lnorm{u}^2:=\int_E|u(e)|^2\,\nu(\diff e);
\]
$L^2(\nu;\R^m)$ is henceforth abbreviated $\Lnu$.

We seek a solution tuple
\[
\Theta_t=(X_t,Y_t,Z_t,U_t);
\]
the jump coefficient is evaluated at the predictable version
$\Theta_t^-:=(X_{t-},Y_{t-},Z_t,U_t)$, which takes left limits in the forward and backward
states ($Z,U$ are already predictable). The coefficients are evaluated along the joint-law flows
\[
\mu_t=\Law(\Theta_t),\qquad \mu_t^-:=\Law(\Theta_t^-)\in\mathcal P_2(\mathcal H).
\]
The distinguishing feature is the law argument:
the coefficients are Lipschitz in the quadratic Wasserstein distance $W_2$
on $\mathcal P_2(\mathcal H)$, so the law of the $\Lnu$-valued jump integrand $U$ itself
may enter the coefficients; both the Lipschitz and the monotonicity hypotheses are
imposed only along the solution's own tuple--law pairs $(\Theta,\Law(\Theta))$.

This class
of equations is motivated by the optimality systems of mean-field games and mean-field
control problems with jumps, where the law argument is typically the state or control law,
and by their numerical analysis: the a~posteriori error estimates of Reisinger, Stockinger
and Zhang~\cite{RSZ} for Brownian MV-FBSDEs rest on the Brownian case of the monotonicity
hypothesis used below, and carrying them to jumps needs the well-posedness theory proved
here.

The primary comparison is the theorem of Li and Min~\cite{LiMin}: global well-posedness
for fully coupled mean-field FBSDEs with jumps, proved by continuation under a
jump-extended $G$-monotonicity condition~\cite[(H3.2)]{LiMin}. Their mean-field
interaction is of expectation-functional (independent-copy) type, their jump integrand
enters only through its pointwise values, and their monotonicity hypothesis imposes
finite total jump mass; the present theorem removes all three restrictions, allowing
$W_2$-dependence on the full tuple law, the law of $U$ included, at arbitrary
$\sigma$-finite activity, under the expected $G$-monotonicity condition set out
below. Example~\ref{ex:expfunc} details the embedding of their interaction class.

For Brownian noise, well-posedness of fully coupled mean-field FBSDEs under joint-law
dependence has been studied by Bensoussan, Yam and Zhang~\cite{BYZ}, Chen and
Nie~\cite{ChenNie} and
others~\cite{CarmonaDelarueFBSDE,ChenNieWang,HuaLuo,MinPengQin,TianYu,WuHao,ZhangHuang}
under various monotonicity and weak-coupling assumptions. All of these are Brownian: no
jump channel, no integrand $U$.

With jumps but no mean-field coupling, fully coupled FBSDEs are well posed over arbitrary
horizons by the same monotone-continuation method: Wu~\cite{WuFBSDEP} for the
Brownian--Poisson case, and domination-monotonicity continuation carried to the
L\'evy-driven case by Xu, Tang and Meng~\cite{XuTangMeng}. For jump \emph{mean-field}
systems in decoupled form, the well-posedness and regularity theory is due to
Li~\cite{LiSPA18}. Neither setting addresses the fully coupled problem with full-tuple
law dependence.

Further fully coupled jump mean-field results trade scope for other restrictions.
Matoussi, Manai and Salhi~\cite{MMS} treat $W_2$-dependence on the $(X,Y)$-marginal law
at finite activity, under monotonicity hypotheses and explicit size restrictions on the
Lipschitz constants.
Chen, Shi and Wu~\cite{ChenShiWu} work on an arbitrary fixed horizon in the weakly
coupled regime, where quantitative weak coupling replaces the structural dissipativity
used here. Liu and Zhang~\cite{LiuZhang} allow pointwise dependence on $Y,Z$ but take the
measure argument to be $\Law(X)$, on a small time interval and under finite jump
activity. In a control direction, mean-field-type control driven by jump-diffusions is
solved globally in time by Bensoussan, Huang, Tang and Yam~\cite{BHTY}, with state-law
dependence and finite total jump mass.

Assumption~\ref{ass:mono} below adjoins a jump term to the monotonicity hypothesis
of~\cite[(H.1)(1)]{RSZ}, posed there, as in the condition~\cite[(A1)]{BYZ}, on the joint
law of $(X,Y,Z)$ with a rectangular full-rank $G$. The measure argument is correspondingly the joint law of the full tuple
$(X,Y,Z,U)$, so the law of the $\Lnu$-valued jump integrand may enter the coefficients, and
$\nu$ is $\sigma$-finite, infinite activity included.

The works surveyed above are each restricted
on at least one of the three axes (measure argument, jump activity, coupling strength). We
prove well-posedness under the stated $L^2$ and monotonicity assumptions by global-in-time
monotone continuation; concretely:
\begin{enumerate}[label=(\roman*)]
\item an a~priori stability estimate: the full difference of two solutions is
  controlled by the differences of their coefficients and data (Theorem~\ref{thm:stab});
\item uniqueness in the natural class (Theorem~\ref{thm:uniq});
\item global existence by monotone continuation from a small-coupling base case
  (Theorem~\ref{thm:cwp});
\item robustness: well-posedness survives full-tuple-law perturbations below the
  dissipativity margin at equal forward and backward dimension
  (Corollary~\ref{cor:robust}; Remark~\ref{rem:generalrank} for general rank); and
\item strictness: expected diagonal monotonicity is strictly weaker than its pointwise
  form (Proposition~\ref{prop:strict}).
\end{enumerate}

The paper is organised as follows. Section~\ref{sec:main-assumptions} sets out the model
and its assumptions. Section~\ref{sec:main} states the main results.
Section~\ref{sec:wp-proof} carries out the continuation proof.
Section~\ref{sec:examples} develops the examples. Section~\ref{sec:conclusion} concludes
and identifies open extensions.
Appendix~\ref{app:lawflow} constructs the Borel law flow used throughout.

 \section{The Brownian--Poisson MV-FBSDEJ}\label{sec:main-assumptions}
Let $(\Omega,\F,(\F_t),\Prob)$ be a filtered probability space satisfying the usual
conditions, over a fixed horizon $T>0$ and dimensions $n,m,d\ge1$, carrying a
$d$-dimensional Brownian motion $W$ and an independent Poisson random measure $N$ on
$\R_+\times E$ with intensity $\diff t\,\nu(\diff e)$, where $\nu$ is a
$\sigma$-finite measure on a countably generated $(E,\mathcal E)$ (so that
$L^2(\nu;\R^j)$ is separable for every $j\ge1$). The filtration $(\F_t)$ is the usual augmentation of the one generated by $W$,
$N$ and an independent initial $\sigma$-field $\F_0$.
We do not require $\nu$ to be finite: the infinite-activity regime,
in which $N$ has infinitely many atoms on every nondegenerate time interval, is admitted
throughout, with the finite-activity case $\nu(E)<\infty$ recovered as a special instance.

Write $\Nt(\diff t,\diff e):=N(\diff t,\diff e)-\diff t\,\nu(\diff e)$ for the
associated compensated measure. For predictable $\psi\ge0$ the compensation formula
reads
\begin{equation}\label{eq:compid}
\E\!\int_0^T\!\!\int_E\psi_s(e)\,N(\diff s,\diff e)
=\E\!\int_0^T\!\!\int_E\psi_s(e)\,\nu(\diff e)\,\diff s\in[0,\infty],
\end{equation}
and the zero-mean property of the compensated integral reads
\begin{equation}\label{eq:zeromean}
\E\!\int_0^T\!\!\int_E\psi_s(e)\,\Nt(\diff s,\diff e)=0
\qquad\text{whenever}\quad
\E\!\int_0^T\!\!\int_E|\psi_s(e)|\,\nu(\diff e)\,\diff s<\infty .
\end{equation}

The unknown is the tuple $\Theta=(X,Y,Z,U)$, with law $\mu_t=\Law(\Theta_t)$ and predictable
counterparts $\Theta_t^-,\mu_t^-$ introduced below; the continuous MV-FBSDEJ is the system
\eqref{eq:mvfbsdej} displayed in the introduction. The tuple
$\Theta_t=(X_t,Y_t,Z_t,U_t)$ takes values in the Hilbert space $\mathcal H$,
componentwise
\[
X_t\in\R^n,\qquad Y_t\in\R^m,\qquad Z_t\in\R^{m\times d},\qquad U_t\in L^2(\nu;\R^m).
\]
The jump terms are evaluated at the predictable tuple
\[
\Theta_t^-=(X_{t-},Y_{t-},Z_t,U_t),\qquad \mu_t^-=\Law(\Theta_t^-),
\]
in which the forward and backward states are replaced by their left limits while the
integrands $Z_t,U_t$, already predictable, are left untouched.
Throughout, $\mu_t$ and $\mu_t^-$ denote the Borel
representatives fixed in Lemma~\ref{lem:lawflow}, which agree with $\Law(\Theta_t)$ and
$\Law(\Theta_t^-)$ for a.e.\ $t$. Expressions such as $b(t,\Theta_t,\mu_t)$ are thereby
defined at every $t$, while every claim made about them is $\diff t$-a.e.\ or
time-integrated.

The filtration generated by $W$ and $N$ has the \emph{Brownian--Poisson predictable
representation property}~(\cite[\S5.3]{Applebaum}; \cite[Theorem~13.49]{HeWangYan}), and it
passes to the independent initial enlargement by $\F_0$ in the standard way:
every square-integrable $(\F_t)$-martingale $M$ can be written as
\[
M_t=M_0+\int_0^t Z_s\,\diff W_s+\int_0^t\!\!\int_E U_s(e)\,\Nt(\diff s,\diff e),
\qquad\text{for every }t\in[0,T],
\]
for an $\F_0$-measurable $M_0$ and a predictable square-integrable pair $(Z,U)$, and this
pair is unique up to null sets.

\begin{definition}[Diagonal input]\label{def:diag}
Fix $t\in[0,T]$. A \emph{diagonal input at time $t$} is a pair $(\Theta,\mu)$ consisting of a
square-integrable $\F_t$-measurable tuple $\Theta\in L^2(\Omega,\F_t;\mathcal H)$, on the
standing basis, \emph{together with its own law} $\mu=\Law(\Theta)\in\mathcal P_2(\mathcal H)$;
such pairs form the diagonal
\[
\mathcal D_t:=\bigl\{\bigl(\Theta,\Law(\Theta)\bigr):\Theta\in L^2(\Omega,\F_t;\mathcal H)\bigr\}
\ \subset\ L^2(\Omega,\F_t;\mathcal H)\times\mathcal P_2(\mathcal H).
\]
\end{definition}

\noindent All hypotheses below are imposed along $\mathcal D_t$ only: the measure argument is
never varied independently of the tuple. Throughout, a lower-case $\theta$ denotes a \emph{point} of
$\mathcal H$ and an upper-case $\Theta$ denotes a \emph{random} $\mathcal H$-valued tuple.

\smallskip\noindent\textbf{Difference notation.}
For two diagonal inputs $(\Theta,\mu),(\Theta',\mu')$ at the same time $t$ on the same
probability space, $\delta$ denotes the difference
\[
\delta X:=X-X',\qquad
\delta\psi:=\psi(t,\Theta,\mu)-\psi(t,\Theta',\mu'),\quad\psi\in\{b,\sigma,h,f\},
\]
and, with the \emph{induced} terminal laws $\rho:=\Law(X_T)$ and
$\rho':=\Law(X'_T)$,
\[
\delta g_T:=g(X_T,\rho)-g(X'_T,\rho') .
\]

Throughout, $b,\sigma,h,f,g$ are deterministic functions of their arguments; randomness
enters only through the inputs and, in Section~\ref{sec:wp-proof}, through additive
square-integrable forcings (the $\omega$-random coefficients of~\cite{ChenNie,LiMin} are
not pursued here). Their signatures are
\[
\begin{aligned}
b&\colon\ [0,T]\times\mathcal H\times\mathcal P_2(\mathcal H)\ \to\ \R^n,\\
\sigma&\colon\ [0,T]\times\mathcal H\times\mathcal P_2(\mathcal H)\ \to\ \R^{n\times d},\\
h&\colon\ [0,T]\times\mathcal H\times\mathcal P_2(\mathcal H)\ \to\ L^2(\nu;\R^n),\\
f&\colon\ [0,T]\times\mathcal H\times\mathcal P_2(\mathcal H)\ \to\ \R^m,\\
g&\colon\ \R^n\times\mathcal P_2(\R^n)\ \to\ \R^m,
\end{aligned}
\]
the terminal function alone carrying no time argument and depending on the state and its
law. All coefficients are jointly Borel measurable in their arguments; for the jump
coefficient, whose values lie in the fibre, this is taken as Borel measurability of the map
\begin{equation}\label{eq:hborel}
(t,\theta,\mu)\ \longmapsto\ h(t,\theta,\cdot,\mu),
\qquad
[0,T]\times\mathcal H\times\mathcal P_2(\mathcal H)\ \longrightarrow\ L^2(\nu;\R^n).
\end{equation}
Since $L^2(\nu;\R^n)$ is separable, \eqref{eq:hborel} follows from joint pointwise
measurability of $(t,\theta,e,\mu)\mapsto h(t,\theta,e,\mu)$: the latter makes $h$ weakly
measurable by Tonelli, and the Pettis measurability
theorem~\cite[Theorem~1.1.6]{HNVW} then upgrades weak to strong measurability.

For predictable processes we use the Fubini isometry
\begin{gather}
\mathcal P:=\text{predictable $\sigma$-field on }\Omega\times[0,T]\text{ under }\diff t\otimes\diff\Prob,\notag\\
L^2(\mathcal P;\Lnu)\ \cong\ L^2(\mathcal P\otimes\mathcal E)\label{eq:fubiniiso}
\end{gather}
of~\cite[Propositions~1.2.24 and~1.2.25]{HNVW}, which supplies a
$\mathcal P\otimes\mathcal E$-measurable representative, unique up to
$\diff t\otimes\diff\Prob\otimes\nu$-null sets; every Poisson integral below, and every
square-integrable $\Lnu$-valued predictable integrand appearing in one, is read through
such a representative. In particular, a pointwise formula $u(e)$ for $u\in\Lnu$ is
meaningful only when it descends to a well-defined element of $\Lnu$, independent of
the representative of $u$. Consequently, along any diagonal input,
\begin{equation}\label{eq:coefmeas}
\begin{aligned}
(s,\omega)&\mapsto\psi\big(s,\Theta_s(\omega),\mu_s\big),\ \ \psi\in\{b,\sigma,f\},
&&\text{progressively measurable},\\[2pt]
(s,\omega,e)&\mapsto h\big(s,\Theta_{s-}(\omega),e,\mu_s^-\big),
&&\text{predictable}.
\end{aligned}
\end{equation}
The well-definedness of the law flows $t\mapsto\mu_t,\mu_t^-$ and the independence of all
time-integrated coefficient identities from the choice of predictable representatives of
$Z,U$ are established in Lemma~\ref{lem:lawflow} in Appendix~\ref{app:lawflow}.

\begin{assumption}[Integrability at the origin]\label{ass:integ}
$\xi_0\in L^2(\Omega,\F_0;\R^n)$ and, with $\delta_0$ the Dirac mass at the relevant origin,
the origin values satisfy
\[
\begin{aligned}
b(\cdot,0,\delta_0)\ &\in\ L^2(0,T;\R^n),
&\qquad\quad \sigma(\cdot,0,\delta_0)\ &\in\ L^2(0,T;\R^{n\times d}),\\[2.2ex]
h(\cdot,0,\cdot,\delta_0)\ &\in\ L^2\bigl(0,T;L^2(\nu;\R^n)\bigr),
&\qquad\quad f(\cdot,0,\delta_0)\ &\in\ L^2(0,T;\R^m).
\end{aligned}
\]
\end{assumption}
\begin{assumption}[Lipschitz coefficients]\label{ass:lip}
There is $L\ge0$ such that, for a.e.\ $t\in[0,T]$ and all
$(\Theta,\mu),(\Theta',\mu')\in\mathcal D_t$ (Definition~\ref{def:diag}),
\[
|\delta b|+|\delta\sigma|_F+\Lnorm{\delta h}+|\delta f|
\ \le\ L\bigl(|\delta X|+|\delta Y|+|\delta Z|_F+\Lnorm{\delta U}+W_2(\mu,\mu')\bigr)
\]
almost surely; and, for all $X_T,X'_T\in L^2(\Omega,\F_T;\R^n)$,
\begin{equation}\label{eq:lipT}
|\delta g_T|\ \le\ L\bigl(|X_T-X'_T|+W_2(\rho,\rho')\bigr)
\end{equation}
almost surely.
\end{assumption}

\begin{assumption}[Coupling matrix]\label{ass:G}
$G\in\R^{m\times n}$ has full rank. Write $\|G\|$ for its operator norm, so that
$|Gv|\le\|G\|\,|v|$ and $|\Gstar v|\le\|G\|\,|v|$, and $s_{\min}(G)>0$ for its
smallest singular value; set
\[
c_G:=s_{\min}(G)^{-2} .
\]
Then
\[
|x|^2\le c_G|Gx|^2\quad(m\ge n),
\qquad\qquad
|y|^2\le c_G|\Gstar y|^2\quad(n\ge m).
\]
When $\Gstar$ is injective the second bound extends to the matrix- and $\Lnu$-valued
arguments,
\[
|Z|_F^2\le c_G|\Gstar Z|_F^2,
\qquad\qquad
\Lnorm{U}^2\le c_G\Lnorm{\Gstar U}^2 .
\]
\end{assumption}
\begin{assumption}[Terminal and interior jump-extended $G$-monotonicity]\label{ass:mono}
For all $X_T,X'_T\in L^2(\Omega,\F_T;\R^n)$,
\[
\E\inner{\delta g_T}{G\,\delta X_T}\ge\alpha\,\E|G\delta X_T|^2,
\]
and, for a.e.\ $t\in[0,T]$ and all $(\Theta,\mu),(\Theta',\mu')\in\mathcal D_t$,
\begin{multline}\label{eq:mono}
\E\!\Big[\inner{\delta b}{\Gstar\delta Y}+\inner{\delta\sigma}{\Gstar\delta Z}_F
+\inner{-\delta f}{G\delta X}+\!\int_E\!\inner{\delta h(e)}{\Gstar\delta U(e)}\nu(\diff e)\Big]\\
\le-\beta_1\E\big[|\Gstar\delta Y|^2+|\Gstar\delta Z|_F^2+\Lnorm{\Gstar\delta U}^2\big]-\beta_2\E|G\delta X|^2
\end{multline}
The constants satisfy
\[
\alpha,\beta_1,\beta_2\ge0,\qquad \alpha+\beta_1>0,\qquad \beta_1+\beta_2>0,
\]
with $\beta_1>0$ if $m<n$ and $\alpha,\beta_2>0$ if $m>n$.
\end{assumption}
\begin{remark}[Wasserstein bounds]
The pair $(\Theta,\Theta')$ realises a coupling of $(\mu,\mu')$ on the common probability
space, so the $W_2$-infimum~\cite[Definition~6.1]{Villani2009} is bounded by its synchronous transport
cost:
\begin{align}
W_2(\mu,\mu')^2&\le\E\big[|\delta X|^2+|\delta Y|^2+|\delta Z|_F^2+\Lnorm{\delta U}^2\big],
\label{eq:coupling}\\
W_2(\rho,\rho')^2&\le\E|\delta X|^2 ,\label{eq:couplingT}
\end{align}
\eqref{eq:couplingT} being the $X$-marginal case of \eqref{eq:coupling}.

In the other direction, let $K$ be a separable Hilbert space and $\bar\phi:\mathcal H\to K$
Lipschitz. The law functional $\mu\mapsto\int_{\mathcal H}\bar\phi\,\diff\mu$ is then
$W_2$-Lipschitz with the same constant: writing the difference as an integral against an
optimal coupling $\pi$ of $(\mu,\mu')$, the Lipschitz bound and Jensen give
\begin{align}
\Big|\int_{\mathcal H}\bar\phi\,\diff\mu-\int_{\mathcal H}\bar\phi\,\diff\mu'\Big|_K
&=\Big|\int_{\mathcal H\times\mathcal H}\big(\bar\phi(\theta)-\bar\phi(\theta')\big)\,\diff\pi\Big|_K\nonumber\\
&\le\mathrm{Lip}(\bar\phi)\int_{\mathcal H\times\mathcal H}|\theta-\theta'|_{\mathcal H}\,\diff\pi\nonumber\\
&\le\mathrm{Lip}(\bar\phi)\Big(\int_{\mathcal H\times\mathcal H}|\theta-\theta'|_{\mathcal H}^2\,\diff\pi\Big)^{1/2}\nonumber\\
&\le\mathrm{Lip}(\bar\phi)\,W_2(\mu,\mu') .\label{eq:lipfun}
\end{align}
Marginal projection $\mathcal P_2(\mathcal H)\to\mathcal P_2(\Lnu)$ is $1$-Lipschitz for $W_2$
by the same argument: write $\mu^U\in\mathcal P_2(\Lnu)$ for the $U$-marginal of $\mu$. The image
of this optimal $\pi$ under $(\theta,\theta')\mapsto(u,u')$ is a coupling of
$(\mu^U,\mu'^U)$, so bounding $W_2(\mu^U,\mu'^U)$ by its cost gives
\begin{align}
W_2(\mu^U,\mu'^U)^2
&\le\int_{\mathcal H\times\mathcal H}\Lnorm{u-u'}^2\,\diff\pi\nonumber\\
&\le\int_{\mathcal H\times\mathcal H}|\theta-\theta'|_{\mathcal H}^2\,\diff\pi\nonumber\\
&\le W_2(\mu,\mu')^2 .\label{eq:margproj}
\end{align}
\end{remark}
\begin{remark}[The coupling matrix $G$]\label{rem:G-pairing}
The full-rank matrix $G\in\R^{m\times n}$ serves to pair the forward and backward
components when $n\neq m$; its full-rank property ensures the resulting $G$-coercivity
controls the full forward variables when $m\ge n$, or the full backward variables when
$n\ge m$.
\end{remark}
\begin{remark}[Expected diagonal monotonicity]\label{rem:diag-mono}
Assumption~\ref{ass:mono} is imposed \emph{in expectation}. The \emph{pointwise}
$G$-monotonicity condition requires the inequality of \eqref{eq:mono}, without expectations,
at every fixed $\theta,\theta'\in\mathcal H$ and $\mu,\mu'\in\mathcal P_2(\mathcal H)$.
Applying it at
\[
\theta=\Theta(\omega),\qquad \theta'=\Theta'(\omega),\qquad
\mu=\Law(\Theta),\qquad \mu'=\Law(\Theta')
\]
and taking expectations gives Assumption~\ref{ass:mono}, so the
pointwise condition is the stronger one. The converse
fails: the expectation may average sign-indefinite pointwise pairings into a dissipative one
(Proposition~\ref{prop:strict}). Example~\ref{ex:concrete}(i) verifies
Assumption~\ref{ass:mono} exactly, Example~\ref{ex:fulllaw} perturbatively.
\end{remark}

\begin{remark}[Coupling formulation]\label{rem:intrinsic}
For $t>0$ the $\sigma$-field $\F_t$ supports a uniformly distributed variable (a
function of $W^1_t$), so every coupling of two laws in $\mathcal P_2(\mathcal H)$ is
the joint law of a pair of $\F_t$-measurable
inputs, by the kernel representation lemma~\cite[Lemma~4.22]{Kallenberg}. The diagonal hypotheses,
Assumptions~\ref{ass:lip} and~\ref{ass:mono}, are therefore intrinsic: properties of
the coefficients, not of the standing basis. As the interior conditions are imposed
$\diff t$-a.e.\ and the terminal ones at $T>0$, no richness assumption on $\F_0$ is
needed. The proofs use only the expected diagonal form.
\end{remark}
\begin{remark}[The $\beta_1$ dichotomy in the jump channel]\label{rem:beta1-dich}
Taking $\delta X=\delta Y=\delta Z=0$ and $\delta U$ arbitrary in \eqref{eq:mono} reduces the
interior pairing to
\[
\E\int_E\inner{\delta h(e)}{\Gstar\delta U(e)}\,\nu(\diff e)
\le-\beta_1\,\E\Lnorm{\Gstar\delta U}^2 ,
\]
so $\beta_1>0$ forces the jump coefficient $h$ to be strictly $U$-dissipative. A jump
coefficient of the natural additive form $h=h(t,x,e)$ carries no $U$-feedback, so
$\delta h=0$ while $\delta U$ may be chosen with $\Gstar\delta U\neq0$; the displayed bound
then forces $\beta_1\le0$, hence $\beta_1=0$ as the constants are non-negative.
Assumption~\ref{ass:mono} then requires $m\ge n$ with $\alpha,\beta_2>0$, the
$G$-injective branch, and the $m<n$ rank case is available only for $U$-dissipative jump
coefficients.
The restriction is structural: it follows from Assumption~\ref{ass:mono} itself, not from
the estimates that exploit it. Recovering the additive form at $m<n$ under a different
structural condition remains open.
\end{remark}
\begin{remark}[Dissipativity margin]\label{rem:margin}
Say that \eqref{eq:mono} holds \emph{with margin} $c>0$ if
\begin{equation}\label{eq:monomargin}
\begin{aligned}
&\E\!\Big[\inner{\delta b}{\Gstar\delta Y}+\inner{\delta\sigma}{\Gstar\delta Z}_F
+\inner{-\delta f}{G\delta X}+\!\int_E\!\inner{\delta h(e)}{\Gstar\delta U(e)}\nu(\diff e)\Big]\\
&\quad\le-\beta_1\E\big[|\Gstar\delta Y|^2+|\Gstar\delta Z|_F^2+\Lnorm{\Gstar\delta U}^2\big]-\beta_2\E|G\delta X|^2\\
&\qquad\ \ -c\,\E\big[|\Gstar\delta Y|^2+|\Gstar\delta Z|_F^2+\Lnorm{\Gstar\delta U}^2+|G\delta X|^2\big] .
\end{aligned}
\end{equation}
Suppose the interior pairing has been bounded with constants $\bar\beta_1,\bar\beta_2$. For
every $0<c<\min\{\bar\beta_1,\bar\beta_2\}$ the same bound is \eqref{eq:monomargin} with the
reduced constants $\beta_i:=\bar\beta_i-c$ and margin $c$: recording weaker constants turns the
difference into slack, which a perturbation raising the interior pairing by at most the margin
leaves intact. When $m=n$ the margin may instead be recorded in the \emph{unweighted} form
\begin{equation}\label{eq:monomarginu}
\begin{aligned}
&\E\!\Big[\inner{\delta b}{\Gstar\delta Y}+\inner{\delta\sigma}{\Gstar\delta Z}_F
+\inner{-\delta f}{G\delta X}+\!\int_E\!\inner{\delta h(e)}{\Gstar\delta U(e)}\nu(\diff e)\Big]\\
&\quad\le-\beta_1\E\big[|\Gstar\delta Y|^2+|\Gstar\delta Z|_F^2+\Lnorm{\Gstar\delta U}^2\big]-\beta_2\E|G\delta X|^2\\
&\qquad\ \ -c\,\E\big[|\delta Y|^2+|\delta Z|_F^2+\Lnorm{\delta U}^2+|\delta X|^2\big] .
\end{aligned}
\end{equation}
\end{remark}
 \PaperPageBreak
\section{Main results}\label{sec:main}

\PaperRunIn{Solution space}
Write
\[
\begin{aligned}
\Lsp{\R^j}&:=\bigl\{V\in\R^j\ \text{c\`adl\`ag adapted}\ :\ \E\!\int_0^T\!|V_t|^2\,\diff t<\infty\bigr\},\\
\Ssp{\R^j}&:=\bigl\{V\in\R^j\ \text{c\`adl\`ag adapted}\ :\ \E\sup_{0\le t\le T}|V_t|^2<\infty\bigr\},\\
\HW(0,T;\R^{m\times d})&:=\Bigl\{Z\in\R^{m\times d}\ \text{predictable}\ :\ \E\!\int_0^T\!|Z_t|_F^2\,\diff t<\infty\Bigr\},\\
\HN(0,T;L^2(\nu;\R^m))&:=\Bigl\{U\in L^2(\nu;\R^m)\ \text{predictable}\ :\ \E\!\int_0^T\!\Lnorm{U_t}^2\,\diff t<\infty\Bigr\}.
\end{aligned}
\]
Set
\[
\begin{aligned}
\HH&:=\Ssp{\R^n}\times\Ssp{\R^m}\times\HW(0,T;\R^{m\times d})\times\HN(0,T;L^2(\nu;\R^m)),\\
\Hweak&:=\Lsp{\R^n}\times\Lsp{\R^m}\times\HW(0,T;\R^{m\times d})\times\HN(0,T;L^2(\nu;\R^m)),
\end{aligned}
\]
with squared norm
\[
\|\Theta\|_{\HH}^2:=\E\sup_{t\le T}|X_t|^2+\E\sup_{t\le T}|Y_t|^2
+\E\!\int_0^T\!\big(|Z_t|_F^2+\Lnorm{U_t}^2\big)\diff t.
\]
The elements are equivalence classes: c\`adl\`ag components are identified when
indistinguishable, $Z$-components when they agree $\diff t\otimes\diff\Prob$-a.e., and
$U$-components when they agree $\diff t\otimes\diff\Prob\otimes\nu$-a.e. Under this
identification $\|\cdot\|_{\HH}$ is a norm rather than a seminorm, and
$(\HH,\|\cdot\|_{\HH})$ is a Banach space. Since
$\E\int_0^T|V_t|^2\,\diff t\le T\,\E\sup_{t\le T}|V_t|^2$, the inclusion
$\HH\subset\Hweak$ holds.

A \emph{solution} of \eqref{eq:mvfbsdej} is a tuple $\Theta=(X,Y,Z,U)\in\HH$ satisfying both
identities of \eqref{eq:mvfbsdej} almost surely for every $t\in[0,T]$, with $h$ evaluated at
the predictable tuple $(\Theta^-_t,\mu^-_t)$. Both sides of each identity are c\`adl\`ag in
$t$, so the exceptional null set may be taken independent of $t$.

The stability and uniqueness theorems below rest on comparing solutions of
\eqref{eq:mvfbsdej} under two coefficient systems, indexed by $k=1,2$. We write $\delta\,{\cdot}={\cdot}^1-{\cdot}^2$, with
$\mu^k_t:=\Law(\Theta^k_t)$, and evaluate every coefficient difference along solution~$2$:
\[
\begin{aligned}
(\delta\psi)(t)&:=\psi^1(t,\Theta^2_t,\mu^2_t)-\psi^2(t,\Theta^2_t,\mu^2_t),
&&\psi\in\{b,\sigma,f\},\\
(\delta h)(t)&:=h^1(t,\Theta^{2,-}_t,\cdot,\mu^{2,-}_t)-h^2(t,\Theta^{2,-}_t,\cdot,\mu^{2,-}_t),\\
\widehat{\delta g}_T&:=g^1\big(X^2_T,\Law(X^2_T)\big)-g^2\big(X^2_T,\Law(X^2_T)\big).
\end{aligned}
\]

\begin{theorem}[Stability under coefficient and data perturbation]\label{thm:stab}
Let the coefficient system $(b^1,\sigma^1,h^1,f^1,g^1)$, with initial datum $\xi_0^1$,
satisfy Assumptions~\ref{ass:integ}--\ref{ass:mono} with constants $L,\alpha,\beta_1,\beta_2$
and coupling matrix $G$; let $(b^2,\sigma^2,h^2,f^2,g^2)$, with initial datum $\xi_0^2$,
satisfy Assumptions~\ref{ass:integ}--\ref{ass:lip} with its own Lipschitz constant; and let
$\Theta^k\in\HH$ solve \eqref{eq:mvfbsdej} for system~$k$. Then
\begin{equation}\label{eq:stab}
\begin{split}
&\E\sup_{t\le T}|\delta X_t|^2+\E\sup_{t\le T}|\delta Y_t|^2
+\E\!\int_0^T\!\big(|\delta Z_t|_F^2+\Lnorm{\delta U_t}^2\big)\diff t\\
&\quad\le C_{\mathrm{stab}}\Big(\E|\delta\xi_0|^2+\E\big|\widehat{\delta g}_T\big|^2\\
&\qquad\qquad\quad+\E\!\int_0^T\!\big(|(\delta b)(t)|^2+|(\delta\sigma)(t)|_F^2
+\Lnorm{(\delta h)(t)}^2+|(\delta f)(t)|^2\big)\diff t\Big),
\end{split}
\end{equation}
where $C_{\mathrm{stab}}$ depends only on $T,L,\|G\|,c_G,\alpha,\beta_1,\beta_2$ and the universal
Burkholder--Davis--Gundy (BDG) constant $c_{\mathrm{BDG}}$. The dependence is on
system~$1$ alone: system~$2$ enters \eqref{eq:stab} only through its right-hand side.
Neither the dimensions $n,m,d$ nor the total mass $\nu(E)$ appears, so the estimate is
uniform over the infinite-activity regime.
\end{theorem}
\noindent The proof, a reduction to Lemma~\ref{lem:pert} at $\lambda=1$, is in
Section~\ref{sub:stabuniq}.

\noindent If system~$2$ also satisfies Assumptions~\ref{ass:G}--\ref{ass:mono}, with its
own coupling matrix and constants, exchanging the roles gives the companion bound with the
differences evaluated along solution~$1$ and the constant built from system~$2$'s data.
When the two systems share the four interior coefficients, the bound gives continuous
dependence on the initial and terminal data. More minimally, the hypotheses on system~$2$
serve only to make the right-hand side of \eqref{eq:stab} finite: the estimate holds
whenever the displayed differences are square-integrable.

\begin{remark}[Solution class]\label{rem:cwp-class}
Let $\Theta=(X,Y,Z,U)\in\Hweak$ satisfy \eqref{eq:mvfbsdej}. Then $\Theta\in\HH$. Indeed, the coefficients evaluated along $\Theta$ lie in
$L^2(\diff t\otimes\diff\Prob)$ by Assumptions~\ref{ass:integ}--\ref{ass:lip}, so Doob/BDG
applied to the forward equation gives
\[
\E\sup_{t\le T}|X_t|^2<\infty ,
\]
whence $g(X_T,\Law(X_T))\in L^2$ by Assumption~\ref{ass:lip}; taking $t=0$ in the backward
equation gives $Y_0\in L^2$. Subtracting that identity from the backward
equation puts $Y$ in forward form,
\[
Y_t=Y_0-\int_0^t f(s,\Theta_s,\mu_s)\,\diff s+\int_0^t Z_s\,\diff W_s
+\int_0^t\!\!\int_E U_s(e)\,\Nt(\diff s,\diff e),
\]
so Cauchy--Schwarz on the drift and Doob/BDG on the two martingales give
\[
\E\sup_{t\le T}|Y_t|^2<\infty .
\]
\end{remark}

\begin{theorem}[Uniqueness]\label{thm:uniq}
Under Assumptions~\ref{ass:lip}--\ref{ass:mono}, the system \eqref{eq:mvfbsdej} has at
most one solution. More precisely, if $\Theta^1,\Theta^2\in\Hweak$ satisfy
\eqref{eq:mvfbsdej} with the same initial condition, then $\Theta^1=\Theta^2$. Furthermore,
$\Theta^1-\Theta^2\in\HH$.
\end{theorem}
\begin{proof}
Let $\Theta^k=(X^k,Y^k,Z^k,U^k)$, $k=1,2$, be two such tuples, and write
$\delta\Theta_t:=\Theta^1_t-\Theta^2_t$, $\mu^k_t:=\Law(\Theta^k_t)$ and, for
$\psi\in\{b,\sigma,h,f\}$,
\[
\delta\psi_t:=\psi(t,\Theta^1_t,\mu^1_t)-\psi(t,\Theta^2_t,\mu^2_t).
\]

By assumption
\[
\E\!\int_0^T\!\big(|X^k_t|^2+|Y^k_t|^2+|Z^k_t|_F^2+\Lnorm{U^k_t}^2\big)\diff t<\infty,
\qquad k=1,2,
\]
so by Fubini
\[
\E|X^k_t|^2+\E|Y^k_t|^2+\E|Z^k_t|_F^2+\E\Lnorm{U^k_t}^2<\infty
\qquad\text{for a.e.\ }t\in[0,T],
\]
that is, $\Theta^k_t\in L^2(\Omega,\F_t;\mathcal H)$ for a.e.\ $t$. Thus
$(\Theta^k_t,\mu^k_t)\in\mathcal D_t$ for a.e.\ $t$, and Assumption~\ref{ass:lip} is
applicable to the pair $(\Theta^1_t,\mu^1_t)$, $(\Theta^2_t,\mu^2_t)$; moreover
$W_2(\mu^1_t,\mu^2_t)^2\le\E|\delta\Theta_t|_{\mathcal H}^2$ by \eqref{eq:coupling}.
Hence Assumption~\ref{ass:lip} yields
\[
\E\big[|\delta b_t|^2+|\delta\sigma_t|_F^2+\Lnorm{\delta h_t}^2+|\delta f_t|^2\big]
\le C\,\E|\delta\Theta_t|_{\mathcal H}^2
\]
for a.e.\ $t$, where $C$ is independent of $t$. Then
\begin{equation}\label{eq:uniq-coeff}
\E\!\int_0^T\!\big[|\delta b_t|^2+|\delta\sigma_t|_F^2+\Lnorm{\delta h_t}^2
+|\delta f_t|^2\big]\diff t<\infty .
\end{equation}

We next show that $\Theta^1-\Theta^2\in\HH$. Since the two solutions have the same
initial condition, the forward equation is
\[
\delta X_t=\int_0^t\delta b_s\,\diff s+\int_0^t\delta\sigma_s\,\diff W_s
+\int_0^t\!\!\int_E\delta h_s(e)\,\Nt(\diff s,\diff e).
\]
By Cauchy--Schwarz, the BDG inequality and the compensation formula,
\[
\E\sup_{0\le t\le T}|\delta X_t|^2
\le C_T\,\E\!\int_0^T\!\big(|\delta b_s|^2+|\delta\sigma_s|_F^2
+\Lnorm{\delta h_s}^2\big)\diff s<\infty
\]
by \eqref{eq:uniq-coeff}, so $\delta X\in\Ssp{\R^n}$.

For the backward equation,
\[
\delta Y_t=\delta g_T+\int_t^T\delta f_s\,\diff s-\int_t^T\delta Z_s\,\diff W_s
-\int_t^T\!\!\int_E\delta U_s(e)\,\Nt(\diff s,\diff e),
\]
where $\delta g_T:=g(X^1_T,\Law(X^1_T))-g(X^2_T,\Law(X^2_T))$. As the terminal
conditions presuppose $\Law(X^k_T)\in\mathcal P_2(\R^n)$, we have $X^k_T\in L^2$, and
Assumption~\ref{ass:lip} with \eqref{eq:couplingT} gives
$\E|\delta g_T|^2\le C\,\E|\delta X_T|^2<\infty$. Furthermore, \eqref{eq:uniq-coeff}
gives $\E\int_0^T|\delta f_s|^2\,\diff s<\infty$. By Lemma~\ref{lem:bpbsde},
\begin{equation}\label{eq:uniq-bwd}
\E\sup_{t\le T}|\delta Y_t|^2
\le C_{\mathrm{BP}}\Big(\E|\delta g_T|^2+\E\!\int_0^T|\delta f_s|^2\,\diff s\Big)
<\infty ,
\end{equation}
therefore $\delta Y\in\Ssp{\R^m}$. By \eqref{eq:uniq-coeff} and \eqref{eq:uniq-bwd} we
have $\delta\Theta\in\HH$.

Now apply Lemma~\ref{lem:pert} of Section~\ref{sec:wp-proof} to the two solutions at
$\lambda=1$ with $\underline\lambda=1$, zero forcings and the same initial datum:
the right-hand side of \eqref{eq:contpert} vanishes, so
$\|\delta\Theta\|_{\HH}^2\le0$, whence $\Theta^1=\Theta^2$ and the uniqueness follows.
\end{proof}

\begin{remark}[Uniqueness without origin integrability]\label{rem:uniq-no-integ}
Remark~\ref{rem:cwp-class} would place each tuple $\Theta\in\Hweak$ satisfying
\eqref{eq:mvfbsdej} individually in $\HH$ if Assumption~\ref{ass:integ} were imposed. But the uniqueness only requires
Assumptions~\ref{ass:lip}--\ref{ass:mono}, as we only need to show that the difference
$\Theta^1-\Theta^2$ belongs to $\HH$, using only the Lipschitz control of coefficient
differences. However, the existence requires Assumption~\ref{ass:integ}.
\end{remark}

\begin{theorem}[Global well-posedness]\label{thm:cwp}
For every horizon $T>0$, under Assumptions~\ref{ass:integ}--\ref{ass:mono} the system
\eqref{eq:mvfbsdej} admits a solution $\Theta=(X,Y,Z,U)\in\HH$, unique by
Theorem~\ref{thm:uniq} and depending continuously on the coefficients and data in the
sense of the estimate \eqref{eq:stab} of Theorem~\ref{thm:stab}.
\end{theorem}
\noindent The proof, by monotone continuation from a small-coupling base case, is in
Section~\ref{sub:proofcwp}.

\begin{corollary}[Robustness of well-posedness]\label{cor:robust}
Let $m=n$, and let the baseline system $(b_0,\sigma_0,h_0,f_0,g)$, with the initial
datum $\xi_0$, satisfy Assumptions~\ref{ass:integ}--\ref{ass:mono} with margin
$c_{\mathrm{diss}}>0$ in the unweighted form \eqref{eq:monomarginu}. Let the perturbations
$p_b,p_\sigma,p_h,p_f$ satisfy the corresponding origin-integrability and diagonal
Lipschitz conditions of Assumptions~\ref{ass:integ}--\ref{ass:lip}, and, for some
$c_{\mathrm{pert}}\ge0$, for a.e.\ $t\in[0,T]$ and all
$(\Theta,\mu),(\Theta',\mu')\in\mathcal D_t$,
\begin{multline}\label{eq:pertpair}
\E\!\Big[\inner{\delta p_b}{\Gstar\delta Y}+\inner{\delta p_\sigma}{\Gstar\delta Z}_F
+\inner{-\delta p_f}{G\delta X}+\!\int_E\!\inner{\delta p_h(e)}{\Gstar\delta U(e)}\nu(\diff e)\Big]\\
\le c_{\mathrm{pert}}\,\E\big[|\delta X|^2+|\delta Y|^2+|\delta Z|_F^2+\Lnorm{\delta U}^2\big]
\end{multline}
If
$c_{\mathrm{pert}}<c_{\mathrm{diss}}$, the perturbed system
\[
(b,\sigma,h,f,g):=(b_0+p_b,\ \sigma_0+p_\sigma,\ h_0+p_h,\ f_0+p_f,\ g)
\]
satisfies, with the same initial datum $\xi_0$, Assumptions~\ref{ass:integ}--\ref{ass:mono}
with the same $G,\alpha,\beta_1,\beta_2$ and margin $c_{\mathrm{diss}}-c_{\mathrm{pert}}$. By
Theorems~\ref{thm:stab}, \ref{thm:uniq} and~\ref{thm:cwp} it is well-posed.

If, in addition, the perturbations are independent of $\theta$ and $W_2$-Lipschitz,
\begin{equation}\label{eq:pertlip}
|\delta p_b|+|\delta p_\sigma|_F+\Lnorm{\delta p_h}+|\delta p_f|\ \le\ L_p\,W_2(\mu,\mu')
\end{equation}
for a.e.\ $t\in[0,T]$ and all $(\Theta,\mu),(\Theta',\mu')\in\mathcal D_t$, then
\eqref{eq:pertpair} holds with
$c_{\mathrm{pert}}=\|G\|L_p$.
\end{corollary}
\begin{proof}
Origin values add in $L^2$ and Lipschitz constants add, so the perturbed system satisfies
Assumptions~\ref{ass:integ}--\ref{ass:lip}; Assumption~\ref{ass:G} and the terminal
inequality of Assumption~\ref{ass:mono} concern only $G$ and $g$, which are unchanged.
For the interior inequality, $\delta\psi=\delta\psi_0+\delta p_\psi$,
$\psi\in\{b,\sigma,h,f\}$, along any pair of diagonal inputs.
The pairing splits into a baseline bracket and a perturbation bracket, bounded by
\eqref{eq:monomarginu} and \eqref{eq:pertpair} respectively, off the union of their two
Lebesgue-null exceptional time sets:
\[
\begin{aligned}
&\E\!\Big[\inner{\delta b}{\Gstar\delta Y}+\inner{\delta\sigma}{\Gstar\delta Z}_F
+\inner{-\delta f}{G\delta X}+\!\int_E\!\inner{\delta h(e)}{\Gstar\delta U(e)}\nu(\diff e)\Big]\\
&\quad=\E\!\Big[\inner{\delta b_0}{\Gstar\delta Y}+\inner{\delta\sigma_0}{\Gstar\delta Z}_F
+\inner{-\delta f_0}{G\delta X}+\!\int_E\!\inner{\delta h_0(e)}{\Gstar\delta U(e)}\nu(\diff e)\Big]\\
&\qquad+\E\!\Big[\inner{\delta p_b}{\Gstar\delta Y}+\inner{\delta p_\sigma}{\Gstar\delta Z}_F
+\inner{-\delta p_f}{G\delta X}+\!\int_E\!\inner{\delta p_h(e)}{\Gstar\delta U(e)}\nu(\diff e)\Big]\\
&\quad\le-\beta_1\E\big[|\Gstar\delta Y|^2+|\Gstar\delta Z|_F^2+\Lnorm{\Gstar\delta U}^2\big]-\beta_2\E|G\delta X|^2\\
&\qquad-(c_{\mathrm{diss}}-c_{\mathrm{pert}})\,\E\big[|\delta Y|^2+|\delta Z|_F^2+\Lnorm{\delta U}^2+|\delta X|^2\big] ,
\end{aligned}
\]
which is \eqref{eq:monomarginu} with the baseline constants and margin
$c_{\mathrm{diss}}-c_{\mathrm{pert}}$.

We now prove that \eqref{eq:pertlip} implies \eqref{eq:pertpair} with
$c_{\mathrm{pert}}=\|G\|L_p$. Write
\[
D^2:=\E\big[|\delta X|^2+|\delta Y|^2+|\delta Z|_F^2+\Lnorm{\delta U}^2\big].
\]
The differences $\delta p_\psi$,
$\psi\in\{b,\sigma,h,f\}$, are deterministic, so they factor out of the expectations.
Cauchy--Schwarz in each pairing, then \eqref{eq:pertlip} and \eqref{eq:coupling}, give
\[
\begin{aligned}
&\E\!\Big[\inner{\delta p_b}{\Gstar\delta Y}+\inner{\delta p_\sigma}{\Gstar\delta Z}_F
+\inner{-\delta p_f}{G\delta X}+\!\int_E\!\inner{\delta p_h(e)}{\Gstar\delta U(e)}\nu(\diff e)\Big]\\
&\quad\le\|G\|\,D\,\bigl(|\delta p_b|+|\delta p_\sigma|_F+\Lnorm{\delta p_h}+|\delta p_f|\bigr)\\
&\quad\le\|G\|\,D\,L_p\,W_2(\mu,\mu')\\
&\quad\le\|G\|L_p\,D^2 ,
\end{aligned}
\]
which is \eqref{eq:pertpair} with $c_{\mathrm{pert}}=\|G\|L_p$.
\end{proof}
\noindent The margin hypothesis is demanding in the jump channel: by the reduction of
Remark~\ref{rem:beta1-dich}, a margin in \eqref{eq:monomarginu} forces every channel of
the baseline to be strictly dissipative. An additive $h_0=h_0(t,x,e)$ can therefore never serve as a baseline, and
jump feedback entering through finitely many functionals of $u$ supplies no margin
either; a margin needs feedback dissipative on the whole of $\Lnu$, as in
Examples~\ref{ex:concrete}(i) and~\ref{ex:dealer}.

\begin{remark}[General rank]\label{rem:generalrank}
The restriction to $m=n$ reflects only the unweighted bookkeeping: the same splitting of
the pairing into baseline and perturbation brackets, run in the weighted quantities of
\eqref{eq:monomargin}, gives the identical conclusion for arbitrary $m,n$ when the margin
and \eqref{eq:pertpair} are both recorded in those quantities.
\end{remark}

\begin{proposition}[Strictness of expected diagonal monotonicity]\label{prop:strict}
Let $n=m=1$, $d\ge1$ and $G=1$, and fix $\beta_1>0$. For every $0<\varepsilon<\beta_1$,
the drift
\begin{equation}\label{eq:bmean}
b(t,\theta,\mu):=-\beta_1 y+\varepsilon\,g_0\bigl(y-m_Y(\mu)\bigr),\qquad
g_0(r):=\max\{-1,\min\{1,r\}\},
\end{equation}
with $m_Y(\mu):=\int_{\mathcal H}y\,\mu(\diff\theta)$ the $Y$-marginal mean, is
free-measure $W_2$-Lipschitz with constant $\beta_1+\varepsilon$ and satisfies, along
all $(\Theta,\mu),(\Theta',\mu')\in\mathcal D_t$,
\begin{equation}\label{eq:diagdrift}
\E\inner{\delta b}{\delta Y}\ \le\ -(\beta_1-\varepsilon)\,\E|\delta Y|^2 ;
\end{equation}
yet for every $\beta\ge0$ the pointwise bound
$\inner{\delta b}{\delta y}\le-\beta\,|\delta y|^2$ fails at a pair of free inputs with
$\theta\in\operatorname{supp}\mu$ and $\theta'\in\operatorname{supp}\mu'$.
\end{proposition}
\begin{proof}
\emph{Free-measure Lipschitz.} Since $g_0$ is $1$-Lipschitz and by \eqref{eq:lipfun},
\[
\begin{aligned}
|\delta b|&\ \le\ \beta_1|\delta y|
+\varepsilon\bigl|\bigl(y-m_Y(\mu)\bigr)-\bigl(y'-m_Y(\mu')\bigr)\bigr|\\
&\ \le\ \beta_1|\delta y|+\varepsilon\bigl(|\delta y|+|m_Y(\mu)-m_Y(\mu')|\bigr)\\
&\ \le\ (\beta_1+\varepsilon)\bigl(|\delta y|+W_2(\mu,\mu')\bigr).
\end{aligned}
\]

\emph{Expected diagonal bound.} Along diagonal inputs, $m_Y(\mu)=\E Y$ and
$m_Y(\mu')=\E Y'$. Since $g_0$ is $1$-Lipschitz and its arguments differ by
$(Y-\E Y)-(Y'-\E Y')=\delta Y-\E\,\delta Y$, Cauchy--Schwarz and
$\E|\delta Y-\E\,\delta Y|^2=\E|\delta Y|^2-|\E\,\delta Y|^2\le\E|\delta Y|^2$ give
\[
\begin{aligned}
\E\inner{\delta b}{\delta Y}
&\ =\ -\beta_1\E|\delta Y|^2
+\varepsilon\,\E\bigl[\bigl(g_0(Y-\E Y)-g_0(Y'-\E Y')\bigr)\,\delta Y\bigr]\\
&\ \le\ -\beta_1\E|\delta Y|^2
+\varepsilon\bigl(\E|\delta Y-\E\,\delta Y|^2\bigr)^{1/2}\bigl(\E|\delta Y|^2\bigr)^{1/2}\\
&\ \le\ -(\beta_1-\varepsilon)\,\E|\delta Y|^2 ,
\end{aligned}
\]
which is \eqref{eq:diagdrift}.

\emph{Pointwise failure on the supports.} Take $\theta=(0,y,0,0)$ and
$\theta'=(0,y-t,0,0)$ with $t>0$, and the two-point laws
\[
\mu:=\tfrac12\delta_{(0,y,0,0)}+\tfrac12\delta_{(0,y-2,0,0)},\qquad
\mu':=\tfrac12\delta_{(0,y-t,0,0)}+\tfrac12\delta_{(0,y-t+2,0,0)},
\]
so that $\theta\in\operatorname{supp}\mu$, $\theta'\in\operatorname{supp}\mu'$, and
$m_Y(\mu)=y-1$, $m_Y(\mu')=y-t+1$. The two arguments of $g_0$ are then $\pm1$, and
\[
\begin{aligned}
y-m_Y(\mu)&=1, \qquad (y-t)-m_Y(\mu')=-1, \qquad \delta y=t,\\
\delta b&=-\beta_1t+\varepsilon\bigl(g_0(1)-g_0(-1)\bigr)=-\beta_1t+2\varepsilon,\\
\inner{\delta b}{\delta y}&=-\beta_1t^2+2\varepsilon t>0
\qquad\text{for }0<t<2\varepsilon/\beta_1,
\end{aligned}
\]
incompatible with every nonpositive bound $-\beta t^2$, $\beta\ge0$.
\end{proof}

 \section{Proof of well-posedness}
\label{sec:wp-proof}

The route is: backward $L^2$ estimate
(Lemma~\ref{lem:bpbsde}) $\to$ small-coupling base case (Lemma~\ref{lem:small}) $\to$
uniform stability on $[\lambda_*,1]$ (Lemma~\ref{lem:pert}), which at $\lambda=1$ gives
Theorems~\ref{thm:stab} and~\ref{thm:uniq} $\to$ fixed-increment continuation
(Lemma~\ref{lem:cont}) $\to$ Theorem~\ref{thm:cwp}, the continuation being along the
coupling strength, in the tradition of Peng and Wu~\cite{PengWu}. The central a priori
estimate rests on a mixed $G$-energy identity for $\inner{G\,\delta X_t}{\delta Y_t}$,
whose compensated-Poisson cross-variation produces a coercive $\Lnu$ jump term controlled
by the $G$-monotonicity assumption. The continuation is in the coupling strength, not the
horizon: shortening the horizon does not make the frozen-coefficient map contractive,
since the $(z,u)$-dependence of $\sigma,h$ and the full-tuple law coupling enter the
contraction constant with no $T$-factor.

Inside $\diff t$-integrals we do not distinguish $X_t,Y_t$ from $X_{t-},Y_{t-}$:
c\`adl\`ag paths have only countably many jumps on compact intervals, so
$\Theta_t^-=\Theta_t$ $\Prob$-a.s.\ and $\mu_t^-=\mu_t$ for a.e.\ $t$
(Lemma~\ref{lem:lawflow}).

Throughout this section, $C$ denotes a positive constant, not necessarily the same at each
occurrence, depending only on $T,L,\|G\|,c_G,\alpha,\beta_1,\beta_2$ and the BDG constant
$c_{\mathrm{BDG}}$, unless a different dependence is stated; recurring constants are
subscripted ($C_0$, $C_{\mathrm{BP}}$, $C_B$, \dots) and fixed where they are introduced.

It\^o's formula for L\'evy-type stochastic integrals, and the product form obtained by
applying it to the bilinear map $(x,y)\mapsto\inner{Gx}{y}$, are those
of~\cite[\S4.4]{Applebaum}, stated for the Brownian--Poisson setting fixed above. The
maximal estimates use the Burkholder--Davis--Gundy inequality at $p=1$ in
quadratic-variation form (\cite{HeWangYan}, chapter on $H^1$ and BMO); valid for
c\`adl\`ag local martingales, it covers the compensated-Poisson martingales below without
a separate jump statement. By localisation and dominated convergence, a c\`adl\`ag local
martingale whose running supremum is integrable has zero-mean increments (the
\emph{supremum criterion} below).

For $\lambda\in[0,1]$ we work throughout with the \emph{$\lambda$-scaled, forced}
system: for a square-integrable forcing
$q=(q^b,q^\sigma,q^h,q^f,q^g_T)$ with
\[
\begin{aligned}
q^b&\in L^2(\Omega\times[0,T];\R^n),
&\qquad q^\sigma&\in L^2(\Omega\times[0,T];\R^{n\times d}),\\
q^h&\in L^2(\Omega\times[0,T]\times E;\R^n),
&\qquad q^f&\in L^2(\Omega\times[0,T];\R^m),\\
q^g_T&\in L^2(\Omega,\F_T;\R^m),
\end{aligned}
\]
where $q^b,q^\sigma,q^f$ are progressively
measurable and $q^h$ is predictable ($\mathcal P\otimes\mathcal E$-measurable as
in~\eqref{eq:fubiniiso}, square-integrable for $\diff\Prob\otimes\diff t\otimes\nu$),
\begin{equation}\label{eq:forced}
\left\{\;\begin{aligned}
X_t&=\xi_0+\lambda\!\int_0^t b(s,\Theta_s,\mu_s)\,\diff s
+\lambda\!\int_0^t\sigma(s,\Theta_s,\mu_s)\,\diff W_s\\
&\hspace{10.2em}+\lambda\!\int_0^t\!\!\int_E h(s,\Theta_s^-,e,\mu_s^-)\,\Nt(\diff s,\diff e)\\
&\hspace{5.4em}+\int_0^t q_s^b\,\diff s+\int_0^t q_s^\sigma\,\diff W_s
+\int_0^t\!\!\int_E q_s^h(e)\,\Nt(\diff s,\diff e),\\[2.2ex]
Y_t&=\lambda\,g(X_T,\Law(X_T))+\lambda\!\int_t^T f(s,\Theta_s,\mu_s)\,\diff s
+q_T^g+\int_t^T q_s^f\,\diff s\\
&\hspace{12.6em}-\int_t^T Z_s\,\diff W_s-\int_t^T\!\!\int_E U_s(e)\,\Nt(\diff s,\diff e).
\end{aligned}\right.
\end{equation}
Call the system \eqref{eq:forced} \emph{well posed at} $\lambda$ --- written
$\mathsf{WP}(\lambda)$ --- if it has a unique solution in $\HH$ for \emph{every} such
forcing $q$. The unforced target is
$\mathsf{WP}(1)$ with $q\equiv0$, so it suffices to prove that
$\mathsf{WP}(1)$ holds.

By Assumptions~\ref{ass:integ}--\ref{ass:lip}, the coefficient processes~\eqref{eq:coefmeas}
along any $\Theta\in\HH$ are square-integrable: for $\psi\in\{b,\sigma,h,f\}$, in the norms of
Assumption~\ref{ass:lip},
\begin{equation}\label{eq:coeffint}
\E\!\int_0^T\!|\psi(s,\Theta_s,\mu_s)|^2\,\diff s
\le C\,\E\!\int_0^T\!\Big(|X_s|^2+|Y_s|^2+|Z_s|_F^2+\Lnorm{U_s}^2
+|\psi(s,0,\delta_0)|^2\Big)\diff s,
\end{equation}
by that Lipschitz bound together with the coupling estimate \eqref{eq:coupling} at
$\mu'=\delta_0$. For the terminal function the analogous bound
\begin{equation}\label{eq:termg}
\E|g(X_T,\Law(X_T))|^2\le C\big(\E|X_T|^2+|g(0,\delta_0)|^2\big)
\end{equation}
follows in the same way from \eqref{eq:lipT} and \eqref{eq:couplingT}.

\subsection{A Brownian--Poisson \texorpdfstring{$L^2$}{L2} estimate}

\begin{lemma}[Brownian--Poisson $L^2$ estimate]\label{lem:bpbsde}
Let $\delta Y$ be a c\`adl\`ag adapted process and let $\delta Z\in\HW$, $\delta U\in\HN$
satisfy
\begin{equation}\label{eq:dYlin}
\delta Y_t=\delta\xi+\int_t^T\phi_s\,\diff s-\int_t^T\delta Z_s\,\diff W_s-\int_t^T\!\int_E\delta U_s(e)\,\Nt(\diff s,\diff e)
\end{equation}
with $\delta\xi\in L^2(\F_T)$ and $\phi\in L^2(\Omega\times[0,T])$. Then
$\delta Y\in\Ssp{\R^m}$, and there is a constant
$C_{\mathrm{BP}}$, depending only on $T$ and $c_{\mathrm{BDG}}$, with
\[
\E\sup_{t\le T}|\delta Y_t|^2+\E\!\int_0^T\!\big(|\delta Z_s|_F^2+\Lnorm{\delta U_s}^2\big)\diff s
\le C_{\mathrm{BP}}\Big(\E|\delta\xi|^2+\E\!\int_0^T|\phi_s|^2\diff s\Big).
\]
\end{lemma}
\begin{proof}
\emph{Membership.} By \eqref{eq:dYlin} with $\int_t^T=\int_0^T-\int_0^t$,
\[
\sup_{t\le T}|\delta Y_t|\le|\delta\xi|+\int_0^T\!|\phi_s|\,\diff s
+2\sup_{t\le T}\Big|\int_0^t\delta Z_s\,\diff W_s\Big|
+2\sup_{t\le T}\Big|\int_0^t\!\!\int_E\delta U_s(e)\,\Nt(\diff s,\diff e)\Big|.
\]
Both stochastic integrals are true $L^2$-martingales by the It\^o and
compensated-Poisson isometries, as $\delta Z\in\HW$ and $\delta U\in\HN$.
Squaring, Cauchy--Schwarz on the drift term, and Doob's inequality then give
$\E\sup_{t\le T}|\delta Y_t|^2<\infty$, that is, $\delta Y\in\Ssp{\R^m}$.

The $\diff s$- and $\diff W$-integrals of $\delta Y$ are continuous, so
$\delta Y$ jumps only through its compensated Poisson integral: its jumps
$\Delta\,\delta Y_s:=\delta Y_s-\delta Y_{s-}$ satisfy $\Delta\,\delta Y_s=\delta U_s(e)$
at each atom $(s,e)$ of $N$; hence the jump variation
splits through $N=\Nt+\diff s\,\nu$ as
\begin{align*}
\sum_{t<s\le T}|\Delta\,\delta Y_s|^2
&=\int_t^T\!\!\int_E|\delta U_s(e)|^2\,N(\diff s,\diff e)\\
&=\int_t^T\!\!\int_E|\delta U_s(e)|^2\,\Nt(\diff s,\diff e)
+\int_t^T\!\Lnorm{\delta U_s}^2\,\diff s.
\end{align*}
It\^o's formula for
$e^{t}|\delta Y_t|^2$ then gives the pathwise identity
\begin{equation}\label{eq:itoY2}
\begin{aligned}
e^{t}|\delta Y_t|^2
&+\int_t^T\!e^{s}\big(|\delta Y_s|^2+|\delta Z_s|_F^2+\Lnorm{\delta U_s}^2\big)\diff s\\
&=e^{T}|\delta\xi|^2+2\int_t^T\!e^{s}\inner{\delta Y_s}{\phi_s}\diff s\\
&\quad-2\int_t^T\!e^{s}\inner{\delta Y_{s-}}{\delta Z_s\,\diff W_s}\\
&\quad-2\int_t^T\!\!\int_E e^{s}\inner{\delta Y_{s-}}{\delta U_s(e)}\,\Nt(\diff s,\diff e)\\
&\quad-\int_t^T\!\!\int_E e^{s}|\delta U_s(e)|^2\,\Nt(\diff s,\diff e).
\end{aligned}
\end{equation}

\emph{Zero-mean terms.} The stochastic-integral terms in \eqref{eq:itoY2} decompose into three
local martingales
\[
\begin{aligned}
M^W_t&:=\int_0^t e^{s}\inner{\delta Y_{s-}}{\delta Z_s\,\diff W_s},\\
M^1_t&:=\int_0^t\!\int_E e^{s}\inner{\delta Y_{s-}}{\delta U_s(e)}\,\Nt(\diff s,\diff e),\\
M^2_t&:=\int_0^t\!\int_E e^{s}|\delta U_s(e)|^2\,\Nt(\diff s,\diff e),
\end{aligned}
\]
each of which we now show satisfies $\E[M_T-M_t]=0$. The three do not share one argument:
$\delta U\in\HN$ gives the quadratic integrand of $M^2$ no square moments, so the $L^2$
theory behind the bounds below does not reach it --- nor, with $\nu(E)=\infty$ admitted,
does localisation along jump times, which then fail to exhaust $[0,T]$. $M^2$ is instead
handled through the compensation formula, which asks only for integrability.
For $M^W$, BDG and Cauchy--Schwarz give
\[
\E\sup_{t\le T}|M^W_t|
\le c_{\mathrm{BDG}}\,e^{T}\big(\E\sup_{t\le T}|\delta Y_t|^2\big)^{1/2}
\Big(\E\!\int_0^T\!|\delta Z_s|_F^2\,\diff s\Big)^{1/2}<\infty ,
\]
and for $M^1$ BDG, Cauchy--Schwarz and \eqref{eq:compid} give
\begin{align*}
\E\sup_{t\le T}|M^1_t|
&\le c_{\mathrm{BDG}}\,\E\Big[\Big(\int_0^T\!\!\int_E e^{2s}|\delta Y_{s-}|^2|\delta U_s(e)|^2
\,N(\diff s,\diff e)\Big)^{1/2}\Big]\\
&\le c_{\mathrm{BDG}}\,e^{T}\big(\E\sup_{t\le T}|\delta Y_t|^2\big)^{1/2}
\Big(\E\!\int_0^T\!\!\int_E|\delta U_s(e)|^2\,N(\diff s,\diff e)\Big)^{1/2}\\
&=c_{\mathrm{BDG}}\,e^{T}\big(\E\sup_{t\le T}|\delta Y_t|^2\big)^{1/2}
\Big(\E\!\int_0^T\!\!\int_E|\delta U_s(e)|^2\,\nu(\diff e)\,\diff s\Big)^{1/2}\\
&=c_{\mathrm{BDG}}\,e^{T}\big(\E\sup_{t\le T}|\delta Y_t|^2\big)^{1/2}
\Big(\E\!\int_0^T\!\Lnorm{\delta U_s}^2\,\diff s\Big)^{1/2}<\infty .
\end{align*}
These bounds put the running suprema of $M^W$ and $M^1$ in $L^1$, so the supremum
criterion gives $\E[M_T-M_t]=0$ for $M\in\{M^W,M^1\}$.
For $M^2$ the integrand $e^{s}|\delta U_s(e)|^2$ is predictable and nonnegative and,
because $\delta U\in\HN$, lies in $L^1(\diff\Prob\otimes\diff s\otimes\nu)$:
\[
\E\!\int_0^T\!\!\int_E e^{s}|\delta U_s(e)|^2\,\nu(\diff e)\,\diff s
\le e^{T}\,\E\!\int_0^T\!\Lnorm{\delta U_s}^2\,\diff s<\infty .
\]
Both the Poisson integral and its compensator are therefore integrable, so the increment
splits under the expectation:
\begin{align*}
\E[M^2_T-M^2_t]
&=\E\!\int_t^T\!\!\int_E e^{s}|\delta U_s(e)|^2\,\Nt(\diff s,\diff e)\\
&=\E\!\int_t^T\!\!\int_E e^{s}|\delta U_s(e)|^2\,N(\diff s,\diff e)
-\E\!\int_t^T\!\!\int_E e^{s}|\delta U_s(e)|^2\,\nu(\diff e)\,\diff s\\
&=0 .
\end{align*}
Taking
expectations in \eqref{eq:itoY2} therefore removes the three martingale terms, and
bounding the drift term by Young's inequality,
$2\inner{\delta Y_s}{\phi_s}\le|\delta Y_s|^2+|\phi_s|^2$, cancels the $|\delta Y_s|^2$
integrals on the two sides:
\[
\E e^{t}|\delta Y_t|^2
+\E\int_t^T e^{s}\big(|\delta Z_s|_F^2+\Lnorm{\delta U_s}^2\big)\diff s
\le
\E e^{T}|\delta\xi|^2
+\E\int_t^T e^{s}|\phi_s|^2\diff s .
\]
Bounding $1\le e^{s}\le e^{T}$ on both sides leaves, for every $t\le T$,
\[
\E|\delta Y_t|^2+\E\int_t^T\big(|\delta Z_s|_F^2+\Lnorm{\delta U_s}^2\big)\diff s\le R ,
\qquad
R:=e^{T}\Big(\E|\delta\xi|^2+\E\int_0^T|\phi_s|^2\diff s\Big).
\]
Both terms on the left are nonnegative and $R$ does not depend on $t$, so each is bounded
by $R$ on its own, the first uniformly in $t$ and the second at $t=0$; adding them,
\begin{equation}\label{eq:supE}
\begin{aligned}
&\sup_{t\le T}\E|\delta Y_t|^2
+\E\int_0^T\big(|\delta Z_s|_F^2+\Lnorm{\delta U_s}^2\big)\diff s\\
&\qquad\le 2R
=C_0\Big(\E|\delta\xi|^2+\E\int_0^T|\phi_s|^2\diff s\Big),
\qquad
C_0:=2e^{T} .
\end{aligned}
\end{equation}

It remains to upgrade $\sup_{t\le T}\E|\delta Y_t|^2$ to $\E\sup_{t\le T}|\delta Y_t|^2$. Taking the difference of
\eqref{eq:dYlin} at times $t$ and $0$ gives the forward form
\[
\delta Y_t=\delta Y_0-\int_0^t\phi_s\diff s+\int_0^t\delta Z_s\diff W_s+\int_0^t\!\int_E\delta U_s(e)\Nt(\diff s,\diff e) .
\]
Squaring, taking the supremum over $t$, Cauchy--Schwarz on the $\diff s$-integral, and BDG on
the two martingales give
\begin{align*}
\E\sup_{t\le T}|\delta Y_t|^2
&\le 4\,\E|\delta Y_0|^2+4\,\E\Big(\int_0^T\!|\phi_s|\diff s\Big)^{2}\\
&\qquad+4\,\E\sup_{t\le T}\Big|\int_0^t\!\delta Z_s\diff W_s\Big|^{2}
+4\,\E\sup_{t\le T}\Big|\int_0^t\!\!\int_E\delta U_s(e)\,\Nt(\diff s,\diff e)\Big|^{2}\\
&\le 4\,\E|\delta Y_0|^2+4T\,\E\!\int_0^T\!|\phi_s|^2\diff s
+4c_{\mathrm{BDG}}\,\E\!\int_0^T\!\big(|\delta Z_s|_F^2+\Lnorm{\delta U_s}^2\big)\diff s .
\end{align*}
Using $\E|\delta Y_0|^2\le\sup_{t\le T}\E|\delta Y_t|^2$ and \eqref{eq:supE},
\begin{align*}
\E\sup_{t\le T}|\delta Y_t|^2
&+\E\!\int_0^T\!\big(|\delta Z_s|_F^2+\Lnorm{\delta U_s}^2\big)\diff s\\
&\le 4\sup_{t\le T}\E|\delta Y_t|^2
+(4c_{\mathrm{BDG}}+1)\,\E\!\int_0^T\!\big(|\delta Z_s|_F^2+\Lnorm{\delta U_s}^2\big)\diff s\\
&\qquad+4T\,\E\!\int_0^T\!|\phi_s|^2\diff s\\
&\le\big(4C_0+(4c_{\mathrm{BDG}}+1)C_0+4T\big)
\Big(\E|\delta\xi|^2+\E\!\int_0^T\!|\phi_s|^2\diff s\Big)\\
&=C_{\mathrm{BP}}\Big(\E|\delta\xi|^2+\E\!\int_0^T\!|\phi_s|^2\diff s\Big),
\qquad C_{\mathrm{BP}}:=C_0(5+4c_{\mathrm{BDG}})+4T .
\end{align*}
This depends only on $T$ and $c_{\mathrm{BDG}}$.
\end{proof}

\subsection{Small-coupling solvability}

\begin{lemma}[Small-coupling base case]\label{lem:small}
Under Assumptions~\ref{ass:integ}--\ref{ass:lip}, there is $\lambda_*\in(0,1]$, depending only on $T$, $L$ and
$c_{\mathrm{BDG}}$, such that $\mathsf{WP}(\lambda)$ holds for
every $\lambda\in[0,\lambda_*]$.
\end{lemma}

\begin{proof}
Fix $\lambda\in[0,1]$ and a forcing $q$. For a trial tuple $\Theta=(X,Y,Z,U)\in\HH$ write
\[
\mu^\Theta_t:=\Law(\Theta_t),\qquad\mu^{\Theta,-}_t:=\Law(\Theta_t^-),
\]
and define $\Phi_\lambda(\Theta):=\bar\Theta=(\bar X,\bar Y,\bar Z,\bar U)$ by the forward and
backward constructions below.

\emph{Forward.} Let $\bar X$ be given by the frozen forward dynamics
\begin{multline}\label{eq:frozenfwd}
\bar X_t=\xi_0+\lambda\!\int_0^t b(s,\Theta_s,\mu^\Theta_s)\,\diff s
+\lambda\!\int_0^t\sigma(s,\Theta_s,\mu^\Theta_s)\,\diff W_s\\
+\lambda\!\int_0^t\!\!\int_E h(s,\Theta_s^-,e,\mu^{\Theta,-}_s)\,\Nt(\diff s,\diff e)\\
+\int_0^t q^b_s\,\diff s+\int_0^t q^\sigma_s\,\diff W_s+\int_0^t\!\!\int_E q^h_s(e)\,\Nt(\diff s,\diff e).
\end{multline}
In \eqref{eq:frozenfwd} the coefficient $h$ is evaluated at the left limits
$(\Theta^-_s,\mu^{\Theta,-}_s)$; by Lemma~\ref{lem:lawflow} these agree with
$(\Theta_s,\mu^\Theta_s)$ outside a $\diff t\otimes\diff\Prob$-null set, so either form may be
used inside the $\diff t\otimes\diff\Prob$-integrals below.
Since $\Theta\in\HH$, \eqref{eq:coeffint} gives
\[
\E\!\int_0^T\!|\psi(s,\Theta_s,\mu^\Theta_s)|^2\,\diff s<\infty
\qquad\text{for }\psi\in\{b,\sigma,h\},
\]
and $q^b,q^\sigma,q^h$ are square-integrable by hypothesis. The coefficients are jointly Borel,
so these integrands are progressively measurable and the jump integrand is predictable.
Cauchy--Schwarz on the drift and Doob/BDG on the two martingale integrals then give
$\E\sup_{t\le T}|\bar X_t|^2<\infty$, so \eqref{eq:frozenfwd} defines a unique
$\bar X\in\Ssp{\R^n}$.

\emph{Backward.} The random variable
\[
\Xi:=\lambda\,g(\bar X_T,\Law(\bar X_T))+q^g_T+\int_0^T(\lambda f(s,\Theta_s,\mu^\Theta_s)+q^f_s)\,\diff s
\]
lies in $L^2(\Omega,\F_T;\R^m)$ by \eqref{eq:termg}, \eqref{eq:coeffint} applied to $f$, and the
square-integrability of $q^g_T$ and $q^f$. Set $M_t:=\E[\Xi\mid\F_t]$. By the Brownian--Poisson
predictable representation property there are unique
$\bar Z\in\HW$, $\bar U\in\HN$ with
\[
M_t=M_0+\int_0^t\bar Z_s\,\diff W_s+\int_0^t\!\int_E\bar U_s(e)\,\Nt(\diff s,\diff e).
\]
Put
\[
\bar Y_t:=M_t-\int_0^t(\lambda f(s,\Theta_s,\mu^\Theta_s)+q^f_s)\,\diff s .
\]
Doob's inequality and $\Xi\in L^2$ give $M\in\Ssp{\R^m}$. The drift is square-integrable by
\eqref{eq:coeffint} and $q^f\in L^2$, so Cauchy--Schwarz gives $\bar Y\in\Ssp{\R^m}$. By construction
$(\bar Y,\bar Z,\bar U)$ solves the frozen linear BSDE
\begin{multline}\label{eq:frozenbwd}
\bar Y_t=\lambda g(\bar X_T,\Law(\bar X_T))+q^g_T
+\int_t^T\!(\lambda f(s,\Theta_s,\mu^\Theta_s)+q^f_s)\,\diff s\\
-\int_t^T\bar Z_s\,\diff W_s-\int_t^T\!\!\int_E\bar U_s(e)\,\Nt(\diff s,\diff e),
\end{multline}
so $\bar\Theta\in\HH$ and $\Phi_\lambda:\HH\to\HH$ is well defined.

\emph{Contraction.} Take $\Theta^1,\Theta^2\in\HH$, write $\bar\Theta^k=\Phi_\lambda(\Theta^k)$,
and set
\[
\begin{gathered}
\delta\psi_s:=\psi(s,\Theta^1_s,\mu^{\Theta^1}_s)-\psi(s,\Theta^2_s,\mu^{\Theta^2}_s)
\quad(\psi\in\{b,\sigma,h,f\}),\\[0.4ex]
\delta X:=X^1-X^2,\qquad \delta\bar X:=\bar X^1-\bar X^2 ,
\end{gathered}
\]
and likewise for the remaining components of $\Theta^1-\Theta^2$ and of
$\bar\Theta^1-\bar\Theta^2$; the common initial datum $\xi_0$ and
the fixed forcing $q$ cancel in every difference.

\emph{Step 1: the coefficient differences.} Write
$S_s:=|\delta b_s|+|\delta\sigma_s|_F+\Lnorm{\delta h_s}+|\delta f_s|$. Assumption~\ref{ass:lip},
applied to the diagonal inputs $(\Theta^1_s,\mu^{\Theta^1}_s)$ and $(\Theta^2_s,\mu^{\Theta^2}_s)$,
gives
\[
S_s\le L\Big(|\delta X_s|+|\delta Y_s|+|\delta Z_s|_F+\Lnorm{\delta U_s}
+W_2\big(\mu^{\Theta^1}_s,\mu^{\Theta^2}_s\big)\Big) .
\]
Squaring, using Cauchy--Schwarz in $\R^5$, taking expectations, and bounding the deterministic
Wasserstein term by \eqref{eq:coupling} give
\[
\E\,S_s^2\le 10L^2\,\E\big[|\delta X_s|^2+|\delta Y_s|^2+|\delta Z_s|_F^2+\Lnorm{\delta U_s}^2\big].
\]
Integrating in $s$ and dominating the $X$- and $Y$-slots by their suprema,
\begin{equation}\label{eq:coeffdiff}
\begin{aligned}
\E\!\int_0^T\! S_s^2\,\diff s
&\le 10L^2\,\E\!\int_0^T\!\big(|\delta X_s|^2+|\delta Y_s|^2+|\delta Z_s|_F^2+\Lnorm{\delta U_s}^2\big)\diff s\\
&\le 10L^2\Big(T\,\E\sup_{t\le T}|\delta X_t|^2+T\,\E\sup_{t\le T}|\delta Y_t|^2\\
&\hspace{6.2em}+\E\!\int_0^T\!\big(|\delta Z_s|_F^2+\Lnorm{\delta U_s}^2\big)\diff s\Big)\\
&\le L_\star^2\,\|\Theta^1-\Theta^2\|_{\HH}^2 ,
\qquad L_\star^2:=10L^2\max\{T,1\} .
\end{aligned}
\end{equation}

\emph{Step 2: the forward difference.} Taking the difference of \eqref{eq:frozenfwd} at
$\Theta^1$ and $\Theta^2$,
\[
\delta\bar X_t=\lambda\!\int_0^t\!\delta b_s\,\diff s+\lambda\!\int_0^t\!\delta\sigma_s\,\diff W_s
+\lambda\!\int_0^t\!\!\int_E\delta h_s(e)\,\Nt(\diff s,\diff e) .
\]
Squaring, taking the supremum in $t$, and taking expectations,
\begin{equation}\label{eq:fwddiff}
\begin{aligned}
\E\sup_{t\le T}|\delta\bar X_t|^2
&\le 3\lambda^2\Big[\E\Big(\int_0^T\!|\delta b_s|\,\diff s\Big)^{2}
+\E\sup_{t\le T}\Big|\int_0^t\!\delta\sigma_s\,\diff W_s\Big|^{2}\\
&\hspace{13.4em}+\E\sup_{t\le T}\Big|\int_0^t\!\!\int_E\delta h_s(e)\,\Nt(\diff s,\diff e)\Big|^{2}\Big]\\
&\le 3\lambda^2\Big[T\,\E\!\int_0^T\!|\delta b_s|^2\,\diff s
+c_{\mathrm{BDG}}\,\E\!\int_0^T\!|\delta\sigma_s|_F^2\,\diff s
+c_{\mathrm{BDG}}\,\E\!\int_0^T\!\Lnorm{\delta h_s}^2\,\diff s\Big]\\
&\le 3\lambda^2\max\{T,c_{\mathrm{BDG}}\}\,
\E\!\int_0^T\!\big(|\delta b_s|^2+|\delta\sigma_s|_F^2+\Lnorm{\delta h_s}^2\big)\diff s\\
&\le 3\lambda^2\max\{T,c_{\mathrm{BDG}}\}\,\E\!\int_0^T\! S_s^2\,\diff s\\
&\le C_X\lambda^2\,\|\Theta^1-\Theta^2\|_{\HH}^2 ,
\qquad C_X:=3\max\{T,c_{\mathrm{BDG}}\}\,L_\star^2 ,
\end{aligned}
\end{equation}
the jump term by \eqref{eq:compid}, the next step since the summands of $S_s$ are
non-negative, and the last by \eqref{eq:coeffdiff}.

\emph{Step 3: the terminal difference.} Put
\[
\delta\bar\xi:=\lambda\big[g(\bar X^1_T,\Law(\bar X^1_T))-g(\bar X^2_T,\Law(\bar X^2_T))\big].
\]
Then
\begin{equation}\label{eq:terdiff}
\begin{aligned}
\E|\delta\bar\xi|^2
&\le\lambda^2L^2\,\E\big(|\delta\bar X_T|+W_2(\Law(\bar X^1_T),\Law(\bar X^2_T))\big)^2\\
&\le2\lambda^2L^2\big(\E|\delta\bar X_T|^2+W_2(\Law(\bar X^1_T),\Law(\bar X^2_T))^2\big)\\
&\le C_g\lambda^2\,\E|\delta\bar X_T|^2 ,\qquad C_g:=4L^2 ,
\end{aligned}
\end{equation}
the first step by \eqref{eq:lipT} and the last by \eqref{eq:couplingT}.

\emph{Step 4: the backward difference.} Taking the difference of \eqref{eq:frozenbwd} at
$\Theta^1$ and $\Theta^2$, the terms $q^g_T$ and $q^f$ cancelling,
\[
\delta\bar Y_t=\delta\bar\xi+\int_t^T\!\lambda\,\delta f_s\,\diff s
-\int_t^T\!\delta\bar Z_s\,\diff W_s-\int_t^T\!\!\int_E\delta\bar U_s(e)\,\Nt(\diff s,\diff e) ,
\]
which is \eqref{eq:dYlin} with $\delta\xi=\delta\bar\xi$ and $\phi_s=\lambda\,\delta f_s$.
$\delta\bar Y$ is c\`adl\`ag adapted, and the remaining hypotheses hold:
\[
\begin{alignedat}{2}
(\delta\bar Z,\delta\bar U)&\in\HW\times\HN
&\quad&\text{since }\bar\Theta^1,\bar\Theta^2\in\HH,\\
\delta\bar\xi&\in L^2(\F_T) &&\text{by }\eqref{eq:termg},\\
\lambda\,\delta f&\in L^2(\Omega\times[0,T]) &&\text{by }\eqref{eq:coeffdiff},
\text{ as }|\delta f|\le S.
\end{alignedat}
\]
By Lemma~\ref{lem:bpbsde},
\begin{equation}\label{eq:bwddiff}
\begin{aligned}
\E\sup_{t\le T}|\delta\bar Y_t|^2
&+\E\!\int_0^T\!\big(|\delta\bar Z_s|_F^2+\Lnorm{\delta\bar U_s}^2\big)\diff s\\
&\le C_{\mathrm{BP}}\Big(\E|\delta\bar\xi|^2+\lambda^2\,\E\!\int_0^T\!|\delta f_s|^2\,\diff s\Big)\\
&\le C_{\mathrm{BP}}\Big(C_g\lambda^2\,\E|\delta\bar X_T|^2
+L_\star^2\lambda^2\,\|\Theta^1-\Theta^2\|_{\HH}^2\Big)\\
&\le C_{\mathrm{BP}}\Big(C_g\lambda^2\,\E\sup_{t\le T}|\delta\bar X_t|^2
+L_\star^2\lambda^2\,\|\Theta^1-\Theta^2\|_{\HH}^2\Big)\\
&\le C_{\mathrm{BP}}\big(C_gC_X\lambda^4+L_\star^2\lambda^2\big)\|\Theta^1-\Theta^2\|_{\HH}^2\\
&\le C_Y\lambda^2\,\|\Theta^1-\Theta^2\|_{\HH}^2 ,
\qquad C_Y:=C_{\mathrm{BP}}\big(C_gC_X+L_\star^2\big) ,
\end{aligned}
\end{equation}
the second inequality by \eqref{eq:terdiff} and \eqref{eq:coeffdiff}, the fourth by
\eqref{eq:fwddiff}, and the last by $\lambda^4\le\lambda^2$ for $\lambda\in[0,1]$.

\emph{Step 5: the contraction.} Since
$\Phi_\lambda(\Theta^1)-\Phi_\lambda(\Theta^2)
=(\delta\bar X,\delta\bar Y,\delta\bar Z,\delta\bar U)$,
\begin{equation}\label{eq:contract}
\begin{aligned}
\|\Phi_\lambda(\Theta^1)-\Phi_\lambda(\Theta^2)\|_{\HH}^2
&=\E\sup_{t\le T}|\delta\bar X_t|^2\\
&\quad+\E\sup_{t\le T}|\delta\bar Y_t|^2
+\E\!\int_0^T\!\big(|\delta\bar Z_s|_F^2+\Lnorm{\delta\bar U_s}^2\big)\diff s\\
&\le(C_X+C_Y)\lambda^2\,\|\Theta^1-\Theta^2\|_{\HH}^2\\
&\le C_\Phi\lambda^2\,\|\Theta^1-\Theta^2\|_{\HH}^2 ,
\qquad C_\Phi:=\max\{1,C_X+C_Y\} ,
\end{aligned}
\end{equation}
the first inequality by \eqref{eq:fwddiff} and \eqref{eq:bwddiff}. Here
\[
L_\star=L_\star(T,L),\quad C_X=C_X(T,L,c_{\mathrm{BDG}}),\quad
C_g=C_g(L),\quad C_{\mathrm{BP}}=C_{\mathrm{BP}}(T,c_{\mathrm{BDG}}),
\]
so $C_\Phi=C_\Phi(T,L,c_{\mathrm{BDG}})$. Setting
\[
\lambda_*:=(2\sqrt{C_\Phi})^{-1}\in(0,\tfrac12]
\]
gives $C_\Phi\lambda^2\le\tfrac14$ for $\lambda\in[0,\lambda_*]$, so that \eqref{eq:contract}
reads
\[
\|\Phi_\lambda(\Theta^1)-\Phi_\lambda(\Theta^2)\|_{\HH}
\le\tfrac12\,\|\Theta^1-\Theta^2\|_{\HH},\qquad\lambda\in[0,\lambda_*] .
\]
By the Banach fixed-point theorem, $\Phi_\lambda$ has a unique fixed point $\Theta\in\HH$.
Evaluating $\Phi_\lambda$ at $\Theta$ itself,
\[
\mu^\Theta=\mu,\qquad\bar\Theta=\Theta ,
\]
so \eqref{eq:frozenfwd}--\eqref{eq:frozenbwd} become the forward and backward equations of
\eqref{eq:forced}, which is solved by $\Theta$.

\emph{Conversely.} Let $\Theta\in\HH$ solve \eqref{eq:forced} and evaluate $\Phi_\lambda$
at the trial $\Theta$. Again $\mu^\Theta=\mu$, so \eqref{eq:frozenfwd} and the forward
equation of \eqref{eq:forced} have the same right-hand side, which does not involve
$\bar X$; hence $\bar X=X$. For the backward component set
\[
M'_t:=Y_t+\int_0^t\!\big(\lambda f(s,\Theta_s,\mu_s)+q^f_s\big)\,\diff s ,
\]
so that
\[
M'_T=\Xi,\qquad
\diff M'_t=Z_t\,\diff W_t+\int_E U_t(e)\,\Nt(\diff t,\diff e) .
\]
Since $Z\in\HW$ and $U\in\HN$, $M'$ is a true $L^2$-martingale, so
\[
M'_t=\E[\Xi\mid\F_t]=M_t .
\]
Uniqueness of the predictable representation forces $(\bar Z,\bar U)=(Z,U)$, and
\eqref{eq:frozenbwd} then gives $\bar Y=Y$, so
\[
\bar\Theta=\Theta .
\]
Solutions are therefore exactly the fixed points, and the solution is unique; as $q$ was
arbitrary, $\mathsf{WP}(\lambda)$ holds on $[0,\lambda_*]$.
\end{proof}

\subsection{Continuous perturbation stability}

\begin{lemma}[Continuous perturbation stability along the homotopy]\label{lem:pert}
Let Assumptions~\ref{ass:lip}--\ref{ass:mono} hold, fix $\underline\lambda\in(0,1]$, and
fix $\lambda\in[\underline\lambda,1]$. Let
$\Theta^1,\Theta^2$ solve \eqref{eq:forced} at parameter $\lambda$ with $\F_0$-measurable
initial data $\xi_0^1,\xi_0^2$ and forcings
$q^k=(q^{b,k},q^{\sigma,k},q^{h,k},q^{f,k},q^{g,k}_T)$, $k=1,2$, where, for each $k$,
$X^k,Y^k$ are c\`adl\`ag adapted with
$\E\int_0^T(|X^k_t|^2+|Y^k_t|^2)\,\diff t<\infty$ and $X^k_T\in L^2(\Omega,\F_T;\R^n)$,
$Z^k\in\HW$, $U^k\in\HN$, and $\Theta^1-\Theta^2\in\HH$. Then, writing
$\delta\,{\cdot}={\cdot}^1-{\cdot}^2$,
\begin{multline}\label{eq:contpert}
\E\sup_{t\le T}|\delta X_t|^2+\E\sup_{t\le T}|\delta Y_t|^2
+\E\!\int_0^T\!\big(|\delta Z_t|_F^2+\Lnorm{\delta U_t}^2\big)\diff t\\
\le C_B\Big(\E|\delta\xi_0|^2+\E|\delta q^g_T|^2
+\E\!\int_0^T\!\big(|\delta q^b_t|^2+|\delta q^\sigma_t|_F^2
+\Lnorm{\delta q^h_t}^2+|\delta q^f_t|^2\big)\diff t\Big),
\end{multline}
where $C_B$ depends only on
$T,L,\|G\|,c_G,\alpha,\beta_1,\beta_2,\underline\lambda$ and $c_{\mathrm{BDG}}$.
\end{lemma}
\begin{proof}
The difference equations below involve the forcings only through $\delta q$, which we
therefore write simply as $q$. We define
\[
P:=\E|\delta\xi_0|^2+\E|q^g_T|^2
+\E\!\int_0^T\!\big(|q^b_t|^2+|q^\sigma_t|_F^2+\Lnorm{q^h_t}^2+|q^f_t|^2\big)\diff t .
\]
Set
\[
\delta\psi_t:=\psi(t,\Theta^1_t,\mu^1_t)-\psi(t,\Theta^2_t,\mu^2_t),\qquad
\psi\in\{b,\sigma,h,f\},
\]
where $\mu^k_t:=\Law(\Theta^k_t)$; likewise $\delta g_T$, with $\delta h$ evaluated at the
predictable tuples $(\Theta^{k,-}_t,\mu^{k,-}_t)$. For a.e.\ $t$ each pair
$(\Theta^k_t,\mu^k_t)$ is a diagonal input in the sense of Definition~\ref{def:diag}:
$\Theta^k_t\in L^2(\Omega,\F_t;\mathcal H)$ for a.e.\ $t$ by Fubini, and $\mu^k_t$ is
the Borel law flow of Lemma~\ref{lem:lawflow}.
The difference equations are
\begin{align}
\diff\,\delta X_t&=(\lambda\delta b_t+q^b_t)\,\diff t+(\lambda\delta\sigma_t+q^\sigma_t)\,\diff W_t
+\int_E(\lambda\delta h_t(e)+q^h_t(e))\,\Nt(\diff t,\diff e),\label{eq:dX}\\
\diff\,\delta Y_t&=-(\lambda\delta f_t+q^f_t)\,\diff t+\delta Z_t\,\diff W_t
+\int_E\delta U_t(e)\,\Nt(\diff t,\diff e),\qquad \delta Y_T=\lambda\,\delta g_T+q^g_T.\label{eq:dY}
\end{align}

The five steps run as follows. Steps~1--2 pair $G\,\delta X$ against $\delta Y$ and
insert Assumption~\ref{ass:mono}: this controls the projections $G\delta X$ and
$\Gstar\delta Y,\Gstar\delta Z,\Gstar\delta U$, but only by the data $P$ together with a
free multiple of the full norm. Steps~3--4 record two a~priori estimates, the forward one
bounding $\delta X$ through $(\delta Y,\delta Z,\delta U)$ and the backward one bounding
$(\delta Y,\delta Z,\delta U)$ through $\delta X$; neither closes alone.
Step~5 breaks the circle: full rank makes $\Gstar$ injective
when $n\ge m$ and $G$ injective when $m\ge n$, so the coercive blocks present shed their
projections, and which blocks are present decides whether the forward or the backward
estimate enters first. The free multiple is then chosen small and absorbed.

\emph{Step 1 ($G$-energy identity).} Both $\delta X$ and $\delta Y$ jump only through
their compensated Poisson integrals, with $\Delta\,\delta X_t=(\lambda\delta h_t+q^h_t)(e)$
and $\Delta\,\delta Y_t=\delta U_t(e)$ at each atom $(t,e)$ of $N$; hence the jump
covariation splits through $N=\Nt+\diff t\,\nu$ as
\[
\begin{aligned}
\sum_{t\le T}\inner{G\,\Delta\delta X_t}{\Delta\delta Y_t}
&=\int_0^T\!\!\int_E\inner{G(\lambda\delta h_t+q^h_t)(e)}{\delta U_t(e)}\,N(\diff t,\diff e)\\
&=\int_0^T\!\!\int_E\inner{G(\lambda\delta h_t+q^h_t)(e)}{\delta U_t(e)}\,\Nt(\diff t,\diff e)\\
&+\int_0^T\!\!\int_E\inner{\lambda\delta h_t(e)+q^h_t(e)}{\Gstar\delta U_t(e)}\,\nu(\diff e)\diff t .
\end{aligned}
\]
The continuous covariation of $\delta X$ and $\delta Y$, whose diffusions are
$\lambda\delta\sigma_t+q^\sigma_t$ and $\delta Z_t$, contributes
\[
\int_0^T\inner{\lambda\delta\sigma_t+q^\sigma_t}{\Gstar\delta Z_t}_F\,\diff t .
\]
It\^o's product formula for $\inner{G\,\delta X_t}{\delta Y_t}$, integrated over $[0,T]$ and using
$\delta X_0=\delta\xi_0$, then gives the pathwise identity
\begin{equation}\label{eq:Gid-path}
\begin{aligned}
\inner{G\delta X_T}{\delta Y_T}-\inner{G\,\delta\xi_0}{\delta Y_0}
&=\int_0^T\inner{\lambda\delta b_t+q^b_t}{\Gstar\delta Y_{t-}}\,\diff t\\
&+\int_0^T\inner{\lambda\delta\sigma_t+q^\sigma_t}{\Gstar\delta Z_t}_F\,\diff t\\
&+\int_0^T\!\!\int_E\inner{\lambda\delta h_t(e)+q^h_t(e)}{\Gstar\delta U_t(e)}\,\nu(\diff e)\diff t\\
&+\int_0^T\inner{-\lambda\delta f_t-q^f_t}{G\delta X_{t-}}\,\diff t\\
&+\int_0^T\inner{G\delta X_{t-}}{\delta Z_t\,\diff W_t}\\
&+\int_0^T\inner{G(\lambda\delta\sigma_t+q^\sigma_t)\,\diff W_t}{\delta Y_{t-}}\\
&+\int_0^T\!\!\int_E\inner{G\delta X_{t-}}{\delta U_t(e)}\,\Nt(\diff t,\diff e)\\
&+\int_0^T\!\!\int_E\inner{G(\lambda\delta h_t+q^h_t)(e)}{\delta Y_{t-}}\,\Nt(\diff t,\diff e)\\
&+\int_0^T\!\!\int_E\inner{G(\lambda\delta h_t+q^h_t)(e)}{\delta U_t(e)}\,\Nt(\diff t,\diff e) .
\end{aligned}
\end{equation}
Of the last five terms of \eqref{eq:Gid-path}, the first four are bounded by BDG and
Cauchy--Schwarz, with \eqref{eq:compid} in the jump terms:
\begin{equation}\label{eq:bdg4}
\left\{\;
\begin{aligned}
&\E\sup_{t\le T}\Big|\int_0^t\inner{G\delta X_{s-}}{\delta Z_s\,\diff W_s}\Big|\\
&\hspace{0.14\linewidth}\le c_{\mathrm{BDG}}\|G\|\big(\E\sup_{t\le T}|\delta X_t|^2\big)^{1/2}
\Big(\E\!\int_0^T\!|\delta Z_t|_F^2\,\diff t\Big)^{1/2},\\
&\E\sup_{t\le T}\Big|\int_0^t\inner{G(\lambda\delta\sigma_s+q^\sigma_s)\,\diff W_s}{\delta Y_{s-}}\Big|\\
&\hspace{0.14\linewidth}\le c_{\mathrm{BDG}}\|G\|\big(\E\sup_{t\le T}|\delta Y_t|^2\big)^{1/2}
\Big(\E\!\int_0^T\!|\lambda\delta\sigma_t+q^\sigma_t|_F^2\,\diff t\Big)^{1/2},\\
&\E\sup_{t\le T}\Big|\int_0^t\!\!\int_E\inner{G\delta X_{s-}}{\delta U_s(e)}\,\Nt(\diff s,\diff e)\Big|\\
&\hspace{0.14\linewidth}\le c_{\mathrm{BDG}}\|G\|\big(\E\sup_{t\le T}|\delta X_t|^2\big)^{1/2}
\Big(\E\!\int_0^T\!\Lnorm{\delta U_t}^2\,\diff t\Big)^{1/2},\\
&\E\sup_{t\le T}\Big|\int_0^t\!\!\int_E\inner{G(\lambda\delta h_s+q^h_s)(e)}{\delta Y_{s-}}\,\Nt(\diff s,\diff e)\Big|\\
&\hspace{0.14\linewidth}\le c_{\mathrm{BDG}}\|G\|\big(\E\sup_{t\le T}|\delta Y_t|^2\big)^{1/2}
\Big(\E\!\int_0^T\!\Lnorm{\lambda\delta h_t+q^h_t}^2\,\diff t\Big)^{1/2} .
\end{aligned}
\right.
\end{equation}

Since $\Theta^1-\Theta^2\in\HH$,
\[
\E\sup_{t\le T}|\delta X_t|^2+\E\sup_{t\le T}|\delta Y_t|^2
+\E\!\int_0^T\!\big(|\delta Z_t|_F^2+\Lnorm{\delta U_t}^2\big)\diff t
=\|\Theta^1-\Theta^2\|_{\HH}^2<\infty ,
\]
and by the computation of \eqref{eq:coeffdiff}, with the forcings square-integrable,
\[
\E\!\int_0^T\!\big(|\lambda\delta\sigma_t+q^\sigma_t|_F^2
+\Lnorm{\lambda\delta h_t+q^h_t}^2\big)\diff t<\infty .
\]
The four suprema in \eqref{eq:bdg4} are therefore integrable, so the supremum criterion
applies to those four local martingales. The fifth has predictable
integrand, and by Cauchy--Schwarz
\[
\begin{aligned}
&\E\!\int_0^T\!\!\int_E\big|\inner{\lambda\delta h_t(e)+q^h_t(e)}{\Gstar\delta U_t(e)}\big|\,\nu(\diff e)\diff t\\
&\hspace{0.26\linewidth}\le\|G\|\Big(\E\!\int_0^T\!\Lnorm{\lambda\delta h_t+q^h_t}^2\diff t\Big)^{1/2}
\Big(\E\!\int_0^T\!\Lnorm{\delta U_t}^2\diff t\Big)^{1/2}\\
&\hspace{0.26\linewidth}<\infty .
\end{aligned}
\]
Hence \eqref{eq:zeromean} applies, and all five terms have zero expectation.

Taking expectations and dropping the left limits in the drift multipliers (the
$\diff t$-integral convention above) reduces \eqref{eq:Gid-path} to
\begin{multline}\label{eq:Gid}
\E\inner{G\delta X_T}{\delta Y_T}-\E\inner{G\,\delta\xi_0}{\delta Y_0}
=\E\!\int_0^T\!\Big[\inner{\lambda\delta b_t+q^b_t}{\Gstar\delta Y_t}
+\inner{\lambda\delta\sigma_t+q^\sigma_t}{\Gstar\delta Z_t}_F\\
+\!\int_E\!\inner{\lambda\delta h_t(e)+q^h_t(e)}{\Gstar\delta U_t(e)}\nu(\diff e)
+\inner{-\lambda\delta f_t-q^f_t}{G\delta X_t}\Big]\diff t .
\end{multline}

\emph{Step 2 (insertion of monotonicity).} In the $h$-channel $\delta h_t$ is evaluated at the
predictable tuples $(\Theta^{k,-}_t,\mu^{k,-}_t)$, which agree with $(\Theta^k_t,\mu^k_t)$
for $\diff t$-a.e.\ $t$ by Lemma~\ref{lem:lawflow}. Assumption~\ref{ass:mono}, applied at
the common input pair $(\Theta^1_t,\mu^1_t),(\Theta^2_t,\mu^2_t)$, therefore bounds the
unperturbed part of the integrand of \eqref{eq:Gid}:
\begin{equation}\label{eq:intmono}
\begin{aligned}
&\lambda\E\Big[\inner{\delta b_t}{\Gstar\delta Y_t}+\inner{\delta\sigma_t}{\Gstar\delta Z_t}_F
+\!\int_E\!\inner{\delta h_t(e)}{\Gstar\delta U_t(e)}\nu(\diff e)+\inner{-\delta f_t}{G\delta X_t}\Big]\\
&\qquad\le-\lambda\beta_1\E\big[|\Gstar\delta Y_t|^2+|\Gstar\delta Z_t|_F^2+\Lnorm{\Gstar\delta U_t}^2\big]
-\lambda\beta_2\E|G\delta X_t|^2 .
\end{aligned}
\end{equation}

The terminal side of \eqref{eq:Gid} is bounded by the terminal monotonicity of
Assumption~\ref{ass:mono}:
\begin{equation}\label{eq:termmono}
\begin{aligned}
\E\inner{G\delta X_T}{\delta Y_T}
&=\lambda\E\inner{G\delta X_T}{\delta g_T}+\E\inner{G\delta X_T}{q^g_T}\\
&\ge\lambda\alpha\E|G\delta X_T|^2-|\E\inner{G\delta X_T}{q^g_T}| .
\end{aligned}
\end{equation}
Recall
\[
P=\E|\delta\xi_0|^2+\E|q^g_T|^2
+\E\!\int_0^T\!\big(|q^b_t|^2+|q^\sigma_t|_F^2+\Lnorm{q^h_t}^2+|q^f_t|^2\big)\diff t ,
\]
and set
\[
\begin{aligned}
\mathcal N'&:=\E|\delta X_T|^2+\E|\delta Y_0|^2
+\E\!\int_0^T\!\big(|\delta X_t|^2+|\delta Y_t|^2+|\delta Z_t|_F^2+\Lnorm{\delta U_t}^2\big)\diff t ,\\
\mathcal K&:=\eta\,\mathcal N'+\tfrac{\|G\|^2}{4\eta}\,P ,\qquad \eta>0 .
\end{aligned}
\]
Inserting \eqref{eq:intmono} and \eqref{eq:termmono} into \eqref{eq:Gid} gives
\[
\begin{aligned}
&\lambda\alpha\E|G\delta X_T|^2-|\E\inner{G\delta X_T}{q^g_T}|-\E\inner{G\,\delta\xi_0}{\delta Y_0}\\
&\qquad\le-\lambda\beta_1\E\!\int_0^T\!\big[|\Gstar\delta Y_t|^2+|\Gstar\delta Z_t|_F^2+\Lnorm{\Gstar\delta U_t}^2\big]\diff t
-\lambda\beta_2\E\!\int_0^T\!|G\delta X_t|^2\diff t\\
&\qquad+\E\!\int_0^T\!\Big(\inner{q^b_t}{\Gstar\delta Y_t}+\inner{q^\sigma_t}{\Gstar\delta Z_t}_F
+\inner{-q^f_t}{G\delta X_t}\\
&\qquad\qquad\quad+\int_E\inner{q^h_t(e)}{\Gstar\delta U_t(e)}\,\nu(\diff e)\Big)\diff t .
\end{aligned}
\]
Rearranging, and applying Young's inequality with $|\Gstar v|\le\|G\|\,|v|$ and
$|Gv|\le\|G\|\,|v|$, gives
\begin{equation}\label{eq:coerc}
\begin{aligned}
&\lambda\alpha\E|G\delta X_T|^2+\lambda\beta_2\E\!\int_0^T\!|G\delta X_t|^2\diff t\\
&\quad+\lambda\beta_1\E\!\int_0^T\!\big[|\Gstar\delta Y_t|^2+|\Gstar\delta Z_t|_F^2+\Lnorm{\Gstar\delta U_t}^2\big]\diff t\\
&\qquad\le|\E\inner{G\,\delta\xi_0}{\delta Y_0}|+|\E\inner{G\delta X_T}{q^g_T}|
+\E\!\int_0^T\!\big(|\inner{q^b}{\Gstar\delta Y}|+|\inner{q^\sigma}{\Gstar\delta Z}_F|\big)\diff t\\
&\qquad+\E\!\int_0^T\!\Bigl(\Bigl|\int_E\inner{q^h(e)}{\Gstar\delta U(e)}\,\nu(\diff e)\Bigr|
+|\inner{q^f}{G\delta X}|\Bigr)\diff t\\
&\qquad\le\mathcal K .
\end{aligned}
\end{equation}
Each coercive block present on the left of \eqref{eq:coerc} is therefore at most
$\mathcal K$.

\emph{Step 3 (forward estimate).} Fix $t\le T$ and integrate \eqref{eq:dX} over $[0,s]$,
$s\le t$, with $\delta X_0=\delta\xi_0$. Squaring, taking the supremum over $s$ and
expectations, Cauchy--Schwarz on the drift, BDG on the two martingales and
Assumption~\ref{ass:lip} give
\[
\begin{aligned}
\E\sup_{s\le t}|\delta X_s|^2
&\le C\Big(\E|\delta\xi_0|^2+\E\!\int_0^t\!\big(|\delta b_s|^2+|\delta\sigma_s|_F^2+\Lnorm{\delta h_s}^2\big)\diff s+P\Big)\\
&\le C\Big(\E|\delta\xi_0|^2+\E\!\int_0^t\!\big(|\delta X_s|^2+|\delta Y_s|^2+|\delta Z_s|_F^2+\Lnorm{\delta U_s}^2\big)\diff s+P\Big)\\
&\le C\Big(\int_0^t\!\E\sup_{r\le s}|\delta X_r|^2\,\diff s
+\E\!\int_0^T\!\big(|\delta Y_s|^2+|\delta Z_s|_F^2+\Lnorm{\delta U_s}^2\big)\diff s+P\Big) .
\end{aligned}
\]
The map $t\mapsto\E\sup_{s\le t}|\delta X_s|^2$ is nondecreasing and finite, since
$\Theta^1-\Theta^2\in\HH$, so Gr\"onwall gives
\begin{equation}\label{eq:fwd}
\E\sup_{t\le T}|\delta X_t|^2\le C\Big(\E\!\int_0^T\!\big(|\delta Y_t|^2+|\delta Z_t|_F^2+\Lnorm{\delta U_t}^2\big)\diff t+P\Big).
\end{equation}

\emph{Step 4 (backward estimate).} It\^o's formula applied to $e^{\kappa t}|\delta Y_t|^2$
gives, for $\kappa>0$,
\begin{equation}\label{eq:bwdito}
\begin{aligned}
&\E\,e^{\kappa t}|\delta Y_t|^2
+\E\!\int_t^T\!e^{\kappa s}\big(\kappa|\delta Y_s|^2+|\delta Z_s|_F^2+\Lnorm{\delta U_s}^2\big)\diff s\\
&\qquad\le\E\,e^{\kappa T}|\lambda\delta g_T+q^g_T|^2
+2\,\E\!\int_t^T\!e^{\kappa s}\inner{\delta Y_s}{\lambda\delta f_s+q^f_s}\diff s .
\end{aligned}
\end{equation}
By Assumption~\ref{ass:lip}, the coupling bound \eqref{eq:coupling}, and Young's
inequality, for every $\varepsilon>0$,
\begin{equation}\label{eq:bwdyoung}
2\,\E\inner{\delta Y_s}{\lambda\delta f_s+q^f_s}
\le \frac C\varepsilon\,\E|\delta Y_s|^2+\varepsilon\,\E\big(|\delta Z_s|_F^2+\Lnorm{\delta U_s}^2\big)
+\frac C\varepsilon\,\E\big(|\delta X_s|^2+|q^f_s|^2\big),
\end{equation}
with $C$ independent of $\varepsilon$.

Insert \eqref{eq:bwdyoung} in \eqref{eq:bwdito}:
\[
\begin{aligned}
&\E\,e^{\kappa t}|\delta Y_t|^2
+\E\!\int_t^T\!e^{\kappa s}\Big[\big(\kappa-\tfrac{C}{\varepsilon}\big)|\delta Y_s|^2
+(1-\varepsilon)\big(|\delta Z_s|_F^2+\Lnorm{\delta U_s}^2\big)\Big]\diff s\\
&\qquad\le\E\,e^{\kappa T}|\lambda\delta g_T+q^g_T|^2
+\tfrac{C}{\varepsilon}\,\E\!\int_t^T\!e^{\kappa s}\big(|\delta X_s|^2+|q^f_s|^2\big)\diff s .
\end{aligned}
\]
Take $\varepsilon<1$ and $\kappa>C/\varepsilon$. The $\delta Y$-integral on the left is then
nonnegative and may be dropped, and the weights satisfy $1\le e^{\kappa s}\le e^{\kappa T}$
on $[0,T]$, so
\[
\E|\delta Y_t|^2+\E\!\int_t^T\!\big(|\delta Z_s|_F^2+\Lnorm{\delta U_s}^2\big)\diff s
\le C\Big(\E|\lambda\delta g_T+q^g_T|^2
+\E\!\int_t^T\!\big(|\delta X_s|^2+|q^f_s|^2\big)\diff s\Big).
\]
With $\E|\delta g_T|^2\le C\,\E|\delta X_T|^2$, the forcings in $P$, and
$\int_t^T\le\int_0^T$ on the right, the bound no longer depends on $t$. Taking the
supremum over $t$ in the first term on the left and $t=0$ in the second, there remains
\begin{equation}\label{eq:wgt}
\sup_{t\le T}\E|\delta Y_t|^2
+\E\!\int_0^T\!\big(|\delta Z_t|_F^2+\Lnorm{\delta U_t}^2\big)\diff t
\le C\Big(\E|\delta X_T|^2+\E\!\int_0^T\!|\delta X_t|^2\diff t+P\Big).
\end{equation}
Writing \eqref{eq:dY} forward from $\delta Y_0$, BDG, Cauchy--Schwarz,
Assumption~\ref{ass:lip} with the coupling bound \eqref{eq:coupling}, and \eqref{eq:wgt} give
\[
\begin{aligned}
\E\sup_{t\le T}|\delta Y_t|^2
&\le C\Big(\E|\delta Y_0|^2+\E\!\int_0^T\!|\lambda\,\delta f_s+q^f_s|^2\diff s
+\E\!\int_0^T\!\big(|\delta Z_s|_F^2+\Lnorm{\delta U_s}^2\big)\diff s\Big)\\
&\le C\Big(\E|\delta Y_0|^2
+\E\!\int_0^T\!\big(|\delta X_s|^2+|\delta Y_s|^2+|\delta Z_s|_F^2+\Lnorm{\delta U_s}^2\big)\diff s+P\Big)\\
&\le C\Big(\E|\delta X_T|^2+\E\!\int_0^T\!|\delta X_t|^2\diff t+P\Big) .
\end{aligned}
\]
The first term of \eqref{eq:wgt} may therefore be replaced by
$\E\sup_{t\le T}|\delta Y_t|^2$:
\begin{equation}\label{eq:bwd}
\E\sup_{t\le T}|\delta Y_t|^2+\E\!\int_0^T\!\big(|\delta Z_t|_F^2+\Lnorm{\delta U_t}^2\big)\diff t
\le C\Big(\E|\delta X_T|^2+\E\!\int_0^T\!|\delta X_t|^2\diff t+P\Big).
\end{equation}

\emph{Step 5 (rank-case closure, then absorb).} Write
\[
\mathcal N:=\E\sup_{t\le T}|\delta X_t|^2+\E\sup_{t\le T}|\delta Y_t|^2+\E\!\int_0^T\!\big(|\delta Z_t|_F^2+\Lnorm{\delta U_t}^2\big)\diff t
\]
so that \eqref{eq:contpert} is the claim $\mathcal N\le C_B P$. Then
$\mathcal N'\le(1+2T)\mathcal N$, and by
\eqref{eq:coerc} each coercive block on its left is at most $\mathcal K$.

\emph{Case A ($n\ge m$; active when $m<n$, where $\beta_1>0$).} As $\Gstar$ is injective,
Assumption~\ref{ass:G} gives
\[
|\delta Y|^2\le c_G|\Gstar\delta Y|^2,\qquad
|\delta Z|_F^2\le c_G|\Gstar\delta Z|_F^2,\qquad
\Lnorm{\delta U}^2\le c_G\Lnorm{\Gstar\delta U}^2 .
\]
With $\beta_1>0$ and $\lambda\ge\underline\lambda$, the $\beta_1$-block of \eqref{eq:coerc} yields
\[
\E\!\int_0^T\!\big(|\delta Y_t|^2+|\delta Z_t|_F^2+\Lnorm{\delta U_t}^2\big)\diff t\le\bigl(c_G/(\underline\lambda\beta_1)\bigr)\mathcal K.
\]
Inserting this in \eqref{eq:fwd}, and the result in \eqref{eq:bwd},
\[
\begin{aligned}
\E\sup_{t\le T}|\delta X_t|^2&\le C(\mathcal K+P),\\
\E\sup_{t\le T}|\delta Y_t|^2
+\E\!\int_0^T\!\big(|\delta Z_t|_F^2+\Lnorm{\delta U_t}^2\big)\diff t&\le C(\mathcal K+P),
\end{aligned}
\]
so $\mathcal N\le C(\mathcal K+P)$.

\emph{Case B ($m\ge n$; active when $m>n$, where $\alpha,\beta_2>0$).} As $G$ is injective,
Assumption~\ref{ass:G} gives
\[
|\delta X|^2\le c_G|G\delta X|^2 .
\]
With $\alpha,\beta_2>0$ and $\lambda\ge\underline\lambda$, the $\alpha,\beta_2$-blocks of
\eqref{eq:coerc} yield
\[
\E|\delta X_T|^2+\E\!\int_0^T\!|\delta X_t|^2\diff t\le\bigl(2c_G/(\underline\lambda\min\{\alpha,\beta_2\})\bigr)\mathcal K.
\]
Inserting this in \eqref{eq:bwd}, and the result in \eqref{eq:fwd},
\[
\begin{aligned}
\E\sup_{t\le T}|\delta Y_t|^2
+\E\!\int_0^T\!\big(|\delta Z_t|_F^2+\Lnorm{\delta U_t}^2\big)\diff t&\le C(\mathcal K+P),\\
\E\sup_{t\le T}|\delta X_t|^2&\le C(\mathcal K+P),
\end{aligned}
\]
so again $\mathcal N\le C(\mathcal K+P)$.

When $m=n$, both $G,\Gstar$ are injective, so Case~A is available whenever $\beta_1>0$ and
Case~B whenever $\alpha,\beta_2>0$; Assumption~\ref{ass:mono} provides at least one of the
two, and when both are available we use whichever gives the smaller bound. In every
case
\[
\mathcal N\le C\Big(\eta\mathcal N'+\tfrac{\|G\|^2}{4\eta} P+P\Big)\le C\eta(1+2T)\mathcal N+C\,P;
\]
choosing $\eta$
small enough that $C\eta(1+2T)\le\tfrac12$ absorbs the $\mathcal N$-term (finite,
since $\Theta^1-\Theta^2\in\HH$) and gives
$\mathcal N\le C\,P$, which is \eqref{eq:contpert} with $C_B:=C$.
\end{proof}

\subsection{Stability}\label{sub:stabuniq}

\begin{proof}[Proof of Theorem~\ref{thm:stab}]
Take the system-$1$ coefficients $(b^1,\sigma^1,h^1,f^1,g^1)$ as the baseline and read
$\Theta^2$ as a forced solution of that system: $\Theta^2$ solves \eqref{eq:forced} at
$\lambda=1$ with
\[
q^{\psi,2}:=-(\delta\psi),\qquad \psi\in\{b,\sigma,h,f\},\qquad
q^{g,2}_T:=-\widehat{\delta g}_T ,
\]
while $\Theta^1$ solves it with zero forcing.

These form a square-integrable forcing for \eqref{eq:forced}: as differences of the
coefficient processes \eqref{eq:coefmeas}, the $q^{\psi,2}$ inherit their measurability,
and they are square-integrable by \eqref{eq:coeffint} applied to each system with its own
constants; $q^{g,2}_T\in L^2(\Omega,\F_T;\R^m)$ by \eqref{eq:lipT} with $X^2_T\in L^2$.

Since $q^1=0$, the difference $\delta q=q^1-q^2$ has $\delta q^\psi=(\delta\psi)$ and
$\delta q^g_T=\widehat{\delta g}_T$, so Lemma~\ref{lem:pert} at parameter $\lambda=1$ and
lower bound $\underline\lambda=1$, applied to
system~$1$, gives
\begin{multline*}
\E\sup_{t\le T}|\delta X_t|^2+\E\sup_{t\le T}|\delta Y_t|^2
+\E\!\int_0^T\!\big(|\delta Z_t|_F^2+\Lnorm{\delta U_t}^2\big)\diff t\\
\le C_B\Big(\E|\delta\xi_0|^2+\E\big|\widehat{\delta g}_T\big|^2\\
+\E\!\int_0^T\!\big(|(\delta b)(t)|^2+|(\delta\sigma)(t)|_F^2
+\Lnorm{(\delta h)(t)}^2+|(\delta f)(t)|^2\big)\diff t\Big),
\end{multline*}
which is \eqref{eq:stab} with $C_{\mathrm{stab}}:=C_B$.
\end{proof}

\subsection{The continuation step}

Taking identical sources and the common initial datum in Lemma~\ref{lem:pert} (at
lower bound $\underline\lambda=\lambda_*$) yields
uniqueness at each $\lambda\in[\lambda_*,1]$; Lemma~\ref{lem:small} gives it on
$[0,\lambda_*]$. It remains to prove existence up to $\lambda=1$.

\begin{lemma}[Continuation]\label{lem:cont}
Let Assumptions~\ref{ass:integ}--\ref{ass:mono} hold, and let $\lambda_*$ be the
small-coupling threshold of Lemma~\ref{lem:small}. There is $\zeta_0>0$, depending only on
$T$, $L$, $\|G\|$, $c_G$, $\alpha$, $\beta_1$, $\beta_2$, $\lambda_*$ and $c_{\mathrm{BDG}}$, such that: if
$\lambda_0\in[\lambda_*,1]$ and $\mathsf{WP}(\lambda_0)$ holds, then
$\mathsf{WP}(\lambda_0+\zeta)$ holds for every $\zeta\in[0,\zeta_0]$ with $\lambda_0+\zeta\le1$.
\end{lemma}

\begin{proof}
Fix $\lambda_0$ with $\mathsf{WP}(\lambda_0)$, set $\lambda_1=\lambda_0+\zeta\le1$, and
fix a forcing $q$ for the $\lambda_1$-system. For a trial $\Theta\in\HH$ define the
augmented forcing
\[
R^\psi(\Theta):=q^\psi+\zeta\,\psi(\cdot,\Theta,\mu^\Theta),\ \ \psi\in\{b,\sigma,h,f\},
\qquad R^g_T(\Theta):=q^g_T+\zeta\,g(X_T,\Law(X_T)).
\]
Then $R(\Theta)$ is
again a square-integrable forcing for \eqref{eq:forced}. Its integrability follows from
\eqref{eq:coeffint} and the Lipschitz/coupling bounds, and its $h$-channel is evaluated at
the predictable tuple $(\Theta^-,\mu^{\Theta,-})$, hence predictable.
Since $\mathsf{WP}(\lambda_0)$ holds, let $\Phi(\Theta):=\bar\Theta$ be the unique
solution in $\HH$ of the forced $\lambda_0$-system with forcing $R(\Theta)$.

For $\Theta^1,\Theta^2$, the outputs $\Phi(\Theta^k)$ solve the same $\lambda_0$-system
with sources $R(\Theta^1),R(\Theta^2)$; $q$ cancels in the difference, leaving
\[
R^\psi(\Theta^1)-R^\psi(\Theta^2)=\zeta[\psi(\cdot,\Theta^1,\mu^{\Theta^1})-\psi(\cdot,\Theta^2,\mu^{\Theta^2})].
\]
By Lemma~\ref{lem:pert} at parameter $\lambda_0$ and lower bound
$\underline\lambda=\lambda_*$, and the
Lipschitz bound of Assumption~\ref{ass:lip},
\[
\begin{aligned}
\|\Phi(\Theta^1)-\Phi(\Theta^2)\|_{\HH}^2
&\le C_B\!\sum_\psi\E\!\int_0^T|R^\psi(\Theta^1)-R^\psi(\Theta^2)|^2\diff t
+C_B\,\E|R^g_T(\Theta^1)-R^g_T(\Theta^2)|^2\\
&\le C_B\,C\,\zeta^2\,\|\Theta^1-\Theta^2\|_{\HH}^2 ,\qquad C=C(T,L) .
\end{aligned}
\]
Put $\zeta_0:=\min\{1,(2\sqrt{C_BC})^{-1}\}$, so that for $\zeta\le\zeta_0$
\[
\|\Phi(\Theta^1)-\Phi(\Theta^2)\|_{\HH}\le\tfrac12\,\|\Theta^1-\Theta^2\|_{\HH} .
\]
By the Banach fixed-point theorem $\Phi$ has a unique fixed point $\Theta$, at which
$R^\psi(\Theta)=q^\psi+\zeta\,\psi(\cdot,\Theta,\mu)$. Since $\lambda_0+\zeta=\lambda_1$,
\[
\begin{aligned}
\lambda_0\,\psi(\cdot,\Theta,\mu)+R^\psi(\Theta)&=\lambda_1\,\psi(\cdot,\Theta,\mu)+q^\psi ,
\qquad\psi\in\{b,\sigma,h,f\},\\
\lambda_0\,g(X_T,\Law(X_T))+R^g_T(\Theta)&=\lambda_1\,g(X_T,\Law(X_T))+q^g_T ,
\end{aligned}
\]
so $\Theta$ solves the forced $\lambda_1$-system \eqref{eq:forced} with forcing $q$.

Any solution of the $\lambda_1$-system is a fixed point of the same $\Phi$, which has only
one, so the solution is unique. Hence $\mathsf{WP}(\lambda_1)$, and $\zeta_0$ is
independent of $\lambda_0$ and $q$.
\end{proof}

\subsection{Proof of Theorem~\texorpdfstring{\ref{thm:cwp}}{3.5}}\label{sub:proofcwp}

\begin{proof}
By Lemma~\ref{lem:small}, $\mathsf{WP}(\lambda)$ holds for $\lambda\in[0,\lambda_*]$. With
$\zeta_0$ from Lemma~\ref{lem:cont}, set
\[
\lambda_k:=\min\{\lambda_*+k\zeta_0,\,1\},\qquad k\ge0 ,
\]
so that $\lambda_0=\lambda_*$, $\lambda_{k+1}-\lambda_k\le\zeta_0$ and $\lambda_k=1$ once
$k\ge\lceil(1-\lambda_*)/\zeta_0\rceil$. Lemma~\ref{lem:cont} carries
$\mathsf{WP}(\lambda_k)$ to $\mathsf{WP}(\lambda_{k+1})$ at every step, so
$\mathsf{WP}(1)$ holds; taking $q\equiv0$ there gives a unique solution of
\eqref{eq:mvfbsdej} in $\HH$.
\end{proof}

 \section{Examples}\label{sec:examples}

Theorem~\ref{thm:cwp} permits broad law dependence; the restrictive hypothesis is the
dissipativity of Assumption~\ref{ass:mono}, which in applications is the one to verify
and which cannot be traded for Lipschitz continuity alone. The jump system
of~\cite[Example~3.2]{LiMin} is linear, hence Lipschitz, and is shown there to admit no
adapted solution on the horizon $T=3\pi/4$.

The examples illustrate both the scope of the law dependence and the role of the
dissipativity assumption. Example~\ref{ex:concrete} exhibits two dissipative models
meeting every assumption exactly, both at arbitrary jump activity: (i)~a law-free affine base,
and (ii)~mean-field coefficients.
Examples~\ref{ex:brownian} and~\ref{ex:expfunc} recover the Brownian~\cite{BYZ} and the
expectation-functional~\cite{LiMin} subclasses as special cases.
Example~\ref{ex:fulllaw} then carries the law dependence beyond those subclasses, to genuinely
nonlinear functionals of the full tuple and of the jump integrand, and
Example~\ref{ex:dealer}, a non-perturbative full $U$-law model, is a
mean-field dealer market in which the law of the $\Lnu$-valued jump integrand arises as a
population law on valuation curves and enters the jump response through a monotone
crowding operator, so that well-posedness holds at every interaction strength.

\begin{example}[Dissipative models]\label{ex:concrete}\leavevmode\par\nopagebreak
\emph{(i) A law-free affine base.} Take
\[
\begin{gathered}
n=m=d=1,\qquad G=1,\qquad \nu\ \sigma\text{-finite},\qquad
\xi_0\in L^2(\Omega,\F_0;\R^n),\\
\sigma_c\in\R,\qquad h_c\in\Lnu,\qquad \beta_1,\beta_2>0 .
\end{gathered}
\]
Consider an agent who corrects an exogenous noise baseline through square-integrable
predictable controls $(a^b,a^\sigma,a^h)$, the jump control $a^h$ valued in $\Lnu$,
\[
\diff X_t=a^b_t\,\diff t+(\sigma_c+a^\sigma_t)\,\diff W_t
+\int_E\bigl(h_c(e)+a^h_t(e)\bigr)\,\Nt(\diff t,\diff e),\qquad X_0=\xi_0,
\]
at quadratic cost
\[
\E\Big[\tfrac12|X_T|^2+\int_0^T\Big(\tfrac{\beta_2}2|X_t|^2
+\tfrac1{2\beta_1}\bigl(|a^b_t|^2+|a^\sigma_t|^2+\Lnorm{a^h_t}^2\bigr)\Big)\diff t\Big].
\]
The Hamiltonian
\[
\begin{aligned}
H&=a^b y+(\sigma_c+a^\sigma)\,z+\int_E\bigl(h_c(e)+a^h(e)\bigr)u(e)\,\nu(\diff e)\\
&\qquad+\tfrac{\beta_2}2|x|^2+\tfrac1{2\beta_1}\bigl(|a^b|^2+|a^\sigma|^2+\Lnorm{a^h}^2\bigr)
\end{aligned}
\]
is strictly convex in the controls, with pointwise minimiser
\[
a^{b*}=-\beta_1 y,\qquad a^{\sigma*}=-\beta_1 z,\qquad a^{h*}(e)=-\beta_1 u(e),
\]
so by the stochastic maximum principle for jump diffusions~\cite{TangLi} the Pontryagin
optimality system consists of the closed-loop state dynamics and the adjoint equation
with driver $\partial_xH=\beta_2x$ and terminal value $X_T$; this is
\eqref{eq:mvfbsdej} with the affine coefficients
\[
\begin{aligned}
b&=-\beta_1 y, &\qquad \sigma&=\sigma_c-\beta_1 z, &\qquad h(e)&=h_c(e)-\beta_1 u(e),\\
f&=\beta_2 x, &\qquad g(x,\rho)&=x .
\end{aligned}
\]
Thus $\beta_1$ is the reciprocal of the quadratic actuation cost, $\beta_2$ the running
state weight, the terminal weight is $1$, and $\sigma_c,h_c$ are the uncontrolled baseline.
With these coefficients \eqref{eq:mvfbsdej} reads
\[
\left\{\;\begin{aligned}
X_t&=\xi_0-\beta_1\!\int_0^t\! Y_s\,\diff s+\int_0^t\!(\sigma_c-\beta_1 Z_s)\,\diff W_s
+\int_0^t\!\!\int_E(h_c(e)-\beta_1 U_s(e))\,\Nt(\diff s,\diff e),\\
Y_t&=X_T+\beta_2\!\int_t^T\! X_s\,\diff s-\int_t^T\! Z_s\,\diff W_s
-\int_t^T\!\!\int_E U_s(e)\,\Nt(\diff s,\diff e).
\end{aligned}\right.
\]
The coefficients are affine, hence Lipschitz, and their origin values, the constants
$\sigma_c$ and $h_c\in\Lnu$, are square-integrable in time over the finite horizon, so
Assumptions~\ref{ass:integ} and~\ref{ass:lip} hold with
$L=\max\{\beta_1,\beta_2,1\}$, where the $1$ covers the terminal map $g$. The matrix $G=1$ is
invertible, so Assumption~\ref{ass:G} holds with $c_G=1$. Since $m=n$,
both rank branches of Assumption~\ref{ass:mono} are available.

Substituting the coefficients, the interior and terminal pairings of \eqref{eq:mono} evaluate
exactly to
\[
\begin{aligned}
&\E\big[\inner{\delta b}{\delta Y}+\inner{\delta\sigma}{\delta Z}+\inner{-\delta f}{\delta X}
+\textstyle\int_E\inner{\delta h(e)}{\delta U(e)}\nu(\diff e)\big]\\
&\qquad=-\beta_1\E\big[|\delta Y|^2+|\delta Z|_F^2+\Lnorm{\delta U}^2\big]-\beta_2\E|\delta X|^2,\\
&\E\inner{\delta g_T}{\delta X_T}=\E|\delta X_T|^2 .
\end{aligned}
\]
Assumption~\ref{ass:mono} holds with equality and $\alpha=1$, so Theorem~\ref{thm:cwp}
yields well-posedness.

The identity is sharp for the constants $\beta_1,\beta_2$, so by Remark~\ref{rem:margin} it
carries any margin $0<c_{\mathrm{diss}}<\min\{\beta_1,\beta_2\}$, unweighted here because
$G=1$.

By Remark~\ref{rem:beta1-dich}, \emph{every} instance with $\beta_1>0$ must carry strictly
$U$-dissipative jump feedback. The affine coefficient $h(e)=h_c(e)-\beta_1u(e)$ realises that
dissipativity with equality,
\[
\int_E\inner{\delta h(e)}{\delta U(e)}\,\nu(\diff e)=-\beta_1\Lnorm{\delta U}^2 .
\]

The baseline also realises the example of Proposition~\ref{prop:strict}: replacing its
drift $-\beta_1y$ by \eqref{eq:bmean} leaves \eqref{eq:mono} intact with constants
$(\beta_1-\varepsilon,\beta_2)$, so Theorem~\ref{thm:cwp} continues to apply.

\smallskip
\emph{(ii) Mean-field terms at arbitrary jump activity.} Write $\theta=(x,y,z,u)$ for a point
of $\mathcal H$ and $\pi_\Psi(\theta)$ for its component indexed by
$\Psi\in\{X,Y,Z,U\}$. The marginal
means of $\mu$ and the barycentre of a terminal law $\rho$ are
\[
m_\Psi(\mu):=\int_{\mathcal H}\pi_\Psi(\theta)\,\mu(\diff\theta),\qquad
\mathfrak b(\rho):=\int_{\R^n}x\,\rho(\diff x) .
\]
Take
\[
\begin{aligned}
b&=-2y-m_Y(\mu), &\qquad \sigma&=-2z-m_Z(\mu), &\qquad h(e)&=-2u(e)-m_U(\mu)(e),\\
f&=2x+m_X(\mu), &\qquad g(x,\rho)&=2x+\mathfrak b(\rho) .
\end{aligned}
\]
Keep $n=m=d=1$ and $G=1$ from (i). Set
\[
\delta m_\Psi:=m_\Psi(\mu)-m_\Psi(\mu'),\qquad
\delta m:=\big(\delta m_X,\delta m_Y,\delta m_Z,\delta m_U\big)
=\int_{\mathcal H}\theta\,\diff\mu-\int_{\mathcal H}\theta\,\diff\mu' ,
\]
so that $\delta m$ is a single element of $\mathcal H$.

Both $\delta m$ and $\mathfrak b(\rho)$ are integrals of the identity map, which is
$1$-Lipschitz, so \eqref{eq:lipfun} applies on $\mathcal P_2(\mathcal H)$ and on
$\mathcal P_2(\R^n)$. Cauchy--Schwarz across the four coordinates of $\mathcal H$ then splits
the first bound into its components:
\[
\begin{gathered}
|\delta m|_{\mathcal H}\le W_2(\mu,\mu') ,\qquad
|\mathfrak b(\rho)-\mathfrak b(\rho')|\le W_2(\rho,\rho') ,\\[1.4ex]
|\delta m_X|+|\delta m_Y|+|\delta m_Z|_F+\Lnorm{\delta m_U}
\le 2\,|\delta m|_{\mathcal H}\le 2\,W_2(\mu,\mu') .
\end{gathered}
\]
Thus Assumption~\ref{ass:lip} holds with $L=2$. The coefficients $b,\sigma,h,f$ vanish at $(0,\delta_0)$, so Assumption~\ref{ass:integ} is
immediate.

Along a diagonal pair the measure argument is the law of the realised tuple, so
$\E\,\delta\Psi$ is deterministic and, for $\Psi\in\{X,Y,Z,U\}$,
\[
m_\Psi(\Law(\Theta))=\E\Psi,\qquad
\delta m_\Psi=\E\,\delta\Psi,\qquad
\E\inner{\E\,\delta\Psi}{\delta\Psi}=|\E\,\delta\Psi|^2 .
\]
Substituting the coefficients and applying these identities, each channel contributes its own
mean square,
\[
\begin{aligned}
\E\inner{-\delta f}{\delta X}
&=\E\inner{-2\,\delta X-\E\,\delta X}{\delta X}=-2\,\E|\delta X|^2-|\E\delta X|^2,\\
\E\inner{\delta b}{\delta Y}
&=\E\inner{-2\,\delta Y-\E\,\delta Y}{\delta Y}=-2\,\E|\delta Y|^2-|\E\delta Y|^2,\\
\E\inner{\delta\sigma}{\delta Z}
&=\E\inner{-2\,\delta Z-\E\,\delta Z}{\delta Z}=-2\,\E|\delta Z|_F^2-|\E\delta Z|_F^2,\\
\E\!\int_E\!\inner{\delta h(e)}{\delta U(e)}\,\nu(\diff e)
&=\E\inner{-2\,\delta U-\E\,\delta U}{\delta U}_{\Lnu}
=-2\,\E\Lnorm{\delta U}^2-\Lnorm{\E\delta U}^2 .
\end{aligned}
\]
Since the mean squares are nonnegative, summing gives the interior pairing
\[
\begin{aligned}
&-2\,\E\big[|\delta X|^2+|\delta Y|^2+|\delta Z|_F^2+\Lnorm{\delta U}^2\big]\\
&\qquad-\big(|\E\delta X|^2+|\E\delta Y|^2+|\E\delta Z|_F^2+\Lnorm{\E\delta U}^2\big)\\
&\qquad\le-2\,\E\big[|\delta X|^2+|\delta Y|^2+|\delta Z|_F^2+\Lnorm{\delta U}^2\big]
\end{aligned}
\]
and the terminal pairing
\[
\E\inner{\delta g_T}{\delta X_T}
=\E\inner{2\,\delta X_T+\E\,\delta X_T}{\delta X_T}
=2\,\E|\delta X_T|^2+|\E\delta X_T|^2
\ge2\,\E|\delta X_T|^2 .
\]
Assumption~\ref{ass:mono} therefore holds with $\alpha=\beta_1=\beta_2=2$. The mean-field
coupling costs nothing and supplies extra dissipation.

These are the coefficients of~\cite[Example~3.1]{LiMin}, carried over here to an
arbitrary $\sigma$-finite $\nu$, including $\nu(E)=\infty$: the monotonicity hypothesis of
their well-posedness theorem~\cite[(H3.2) and Theorem~3.1]{LiMin} bounds the total mass of
their L\'evy measure, whereas the verification above uses only $U\in\Lnu$.

\end{example}

\begin{example}[Brownian reduction]\label{ex:brownian}
Remove the Poisson channel from the stochastic basis, so that the filtration is generated by
$W$ and $\F_0$ and the representation property is the Brownian one. Delete the $U$-coordinate
from the tuple, reducing $\mathcal H$ to $\R^n\times\R^m\times\R^{m\times d}$. Merely setting
$h\equiv0$ on a Brownian--Poisson basis is not a reduction, since the backward Poisson integral
and its integrand $U$ persist in the representation.

The assumptions reduce to $W_2$-Lipschitz joint-law dependence of $(X,Y,Z)$ under
$G$-monotonicity, the condition imposed in~\cite[(A1)]{BYZ} with rectangular full-rank $G$.
With no jump channel the $U$-dissipativity constraint of
Remark~\ref{rem:beta1-dich} is vacuous, so both rank branches of Assumption~\ref{ass:mono}
are available.
\end{example}

\begin{example}[Expectation-functional interactions]\label{ex:expfunc}
Let $\bar\phi:[0,T]\times\mathcal H\times\mathcal H\to\R^k$ be measurable and Lipschitz in
$(\theta,\theta')$ uniformly in $t$, and set
\[
\phi(t,\theta,\mu):=\int_{\mathcal H}\bar\phi(t,\theta,\theta')\,\mu(\diff\theta') .
\]
At each fixed $t$ and $\theta$ this is the Bochner mean of the kernel
$\bar\phi(t,\theta,\cdot)$, whose Lipschitz constant is at most $\mathrm{Lip}(\bar\phi)$, so
\eqref{eq:lipfun} supplies the law-dependence component of Assumption~\ref{ass:lip}. Its
state-variable Lipschitz and integrability requirements are imposed on the full coefficient
system.

This class contains the deterministic-coefficient, finite-activity independent-copy
interactions of~\cite{LiMin}. Their L\'evy measure is the $\nu$ of the present setting. Let
$\bar\psi$ be the two-point kernel of their drift, diffusion or driver, and $\bar h$ that of
their jump coefficient. Both are assumed Lipschitz in all arguments except $t$, with
constant uniform in $t$, and with square-integrable origin values,
$\bar\psi(\cdot,0,\dots,0)\in L^2(0,T)$, the counterpart of the origin integrability
imposed in~\cite[(H3.1)]{LiMin}.
The pointwise-jump form embeds via
\[
\bar\phi(t,\theta,\theta')=\int_E\bar\psi\bigl(t,x,y,z,u(e),x',y',z',u'(e)\bigr)\nu(\diff e).
\]
The integral is finite at every $(\theta,\theta')$: the Lipschitz bound on $\bar\psi$ and
Cauchy--Schwarz in $L^2(\nu)$ give
$\int_E\bigl(|u(e)|+|u'(e)|\bigr)\nu(\diff e)\le\nu(E)^{1/2}\bigl(\Lnorm{u}+\Lnorm{u'}\bigr)$,
and the origin term contributes $\nu(E)\,|\bar\psi(t,0,\dots,0)|$. The same bounds give,
with a constant finite when $\nu(E)<\infty$,
\[
\begin{aligned}
&|\bar\phi(t,\theta_1,\theta_1')-\bar\phi(t,\theta_2,\theta_2')|\\
&\qquad\le\mathrm{Lip}(\bar\psi)\max\{\nu(E),\nu(E)^{1/2}\}
\bigl(|\theta_1-\theta_2|_{\mathcal H}+|\theta_1'-\theta_2'|_{\mathcal H}\bigr),
\end{aligned}
\]
so $\bar\phi$ is Lipschitz on $\mathcal H\times\mathcal H$. The pointwise values $u(e),u'(e)$ of the
$L^2(\nu)$-components enter only under the $\nu$-integral, so the functional is
well defined on equivalence classes: changing $u$ on a $\nu$-null set changes the
integrand on a $\nu$-null set of marks.

In their jump coefficient the mark $e$ is the integrator's and is not $\nu$-integrated, so
that coefficient embeds instead through the Nemytskii map
\[
[\mathbf H_t(\theta,\theta')](e):=\bar h\bigl(t,x,y,z,u(e),x',y',z',u'(e),e\bigr)
\quad\text{for $\nu$-a.e.\ }e .
\]
This map is again well defined on equivalence classes. Assume it jointly Borel, and impose the
natural quantitative hypotheses
\[
\begin{aligned}
\|\mathbf H_t(\theta_1,\theta_1')-\mathbf H_t(\theta_2,\theta_2')\|_{L^2(\nu)}
&\le C(|\theta_1-\theta_2|_{\mathcal H}+|\theta_1'-\theta_2'|_{\mathcal H}),\\
\int_0^T\|\mathbf H_t(0,0)\|_{L^2(\nu)}^2\,\diff t&<\infty ,
\end{aligned}
\]
so that $\mathbf H_t$ is Lipschitz into $L^2(\nu;\R^n)$ and \eqref{eq:lipfun} applies,
with that space in the role of $K$, to
\[
h(t,\theta,\mu):=\int_{\mathcal H}\mathbf H_t(\theta,\theta')\,\mu(\diff\theta') .
\]
\end{example}

\begin{example}[Nonlinear law functionals]\label{ex:fulllaw}
Fix
\[
\begin{aligned}
\mu_0&\in\mathcal P_2(\mathcal H), &\quad \Lambda_0&\in\mathcal P_2(\Lnu),
&\quad \varepsilon&>0,\\
\mathfrak q,\tilde{\mathfrak q},\hat{\mathfrak q}&:\R_+\to\R, &\quad Q&:\R^J\to\R,
&\quad \varphi_1,\dots,\varphi_J&:\Lnu\to\R,\\
w_0,w_1&\in\R^m, &\quad v_0&\in\R^n, &\quad \hat h&\in L^2(\nu;\R^n),
\end{aligned}
\]
with $\mathfrak q,\tilde{\mathfrak q},\hat{\mathfrak q},Q$ and the features $\varphi_j$ bounded and Lipschitz, and
$w_0,w_1,v_0,\hat h$ the perturbation directions.

Let the baseline generator $(b_0,\sigma_0,h_0,f_0,g)$ be as in Corollary~\ref{cor:robust}:
$m=n$, so that the full-rank $G$ is invertible; Assumptions~\ref{ass:integ}--\ref{ass:mono}
hold; and there is a dissipativity margin $c_{\mathrm{diss}}>0$ in the unweighted form
\eqref{eq:monomarginu}. Example~\ref{ex:concrete}(i) supplies such a baseline. Perturb by
\begin{align*}
f(t,\theta,\mu)&=f_0(t,\theta)+\varepsilon\,\mathfrak q\bigl(W_2(\mu,\mu_0)\bigr)w_0
+\varepsilon\,Q\biggl(\int_{\Lnu}\varphi_1\,\diff\mu^U,\dots,\int_{\Lnu}\varphi_J\,\diff\mu^U\biggr)w_1,\\
b(t,\theta,\mu)&=b_0(t,\theta)+\varepsilon\,\tilde{\mathfrak q}\bigl(W_2(\mu^U,\Lambda_0)\bigr)v_0,\\
h(t,\theta,e,\mu)&=h_0(t,\theta,e)+\varepsilon\,\hat{\mathfrak q}\bigl(W_2(\mu^U,\Lambda_0)\bigr)\hat h(e) .
\end{align*}
The baseline is taken law-free for simplicity: a law-dependent dissipative baseline
works identically.

The law functionals are of three types:
\begin{enumerate}[label=(\roman*)]
\item $\mathfrak q(W_2(\mu,\mu_0))$, a nonlinear \emph{radial} statistic of the full tuple law:
  a function of its $W_2$-distance to a fixed reference measure alone;
\item $\tilde{\mathfrak q}(W_2(\mu^U,\Lambda_0))$ and $\hat{\mathfrak q}(W_2(\mu^U,\Lambda_0))$, the same for the
  $U$-marginal, the latter entering the \emph{jump coefficient itself}; and
\item the $Q$-term, a nonlinear \emph{cylindrical} functional of the $U$-marginal: finitely
  many integrals of arbitrary bounded-Lipschitz features on the infinite-dimensional fibre
  $\Lnu$, composed through the nonlinear $Q$.
\end{enumerate}
All three lie beyond the expectation and two-point-kernel jump-integrand interactions of the
preceding examples, and beyond the linear jump-integrand expectations of the cited models. All
are $W_2$-Lipschitz in $\mu$. The radial statistic has constant $\mathrm{Lip}(\mathfrak q)$, by the
triangle inequality for $W_2$. The $U$-marginal statistics have constants
$\mathrm{Lip}(\tilde{\mathfrak q})$ and $\mathrm{Lip}(\hat{\mathfrak q})$, since marginal projection is
$1$-Lipschitz for $W_2$ by \eqref{eq:margproj}. Each cylindrical feature is the Bochner mean of
the kernel $\theta'\mapsto\varphi_j(u')$, so \eqref{eq:lipfun} applies and the $Q$-term has
constant
\[
L_Q:=\mathrm{Lip}(Q)\Bigl(\sum_{j=1}^J\mathrm{Lip}(\varphi_j)^2\Bigr)^{1/2} .
\]
The perturbations are bounded multiples of the fixed vectors $w_0,w_1,v_0$ and of
$\hat h\in L^2(\nu;\R^n)$, so they satisfy the origin-integrability condition of
Assumption~\ref{ass:integ}.
They are independent of $\theta$ and satisfy the $W_2$-Lipschitz bound \eqref{eq:pertlip}
with constant
\[
L_p:=\varepsilon\bigl(|w_0|\,\mathrm{Lip}(\mathfrak q)+|v_0|\,\mathrm{Lip}(\tilde{\mathfrak q})+|w_1|\,L_Q
+\Lnorm{\hat h}\,\mathrm{Lip}(\hat{\mathfrak q})\bigr) ,
\]
hence the diagonal Lipschitz condition of Assumption~\ref{ass:lip} as well.

By \eqref{eq:pertlip} and Corollary~\ref{cor:robust}, \eqref{eq:pertpair} holds with
$c_{\mathrm{pert}}=\|G\|L_p$, so the perturbed system is well-posed, with margin
$c_{\mathrm{diss}}-\|G\|L_p$, whenever
\begin{equation}\label{eq:fulllaw-eps}
\varepsilon\,\|G\|\bigl(|w_0|\,\mathrm{Lip}(\mathfrak q)+|v_0|\,\mathrm{Lip}(\tilde{\mathfrak q})+|w_1|\,L_Q
+\Lnorm{\hat h}\,\mathrm{Lip}(\hat{\mathfrak q})\bigr)\ <\ c_{\mathrm{diss}} .
\end{equation}
The smallness in \eqref{eq:fulllaw-eps} binds the law-dependence increment alone; the
forward--backward coupling of the baseline is unrestricted, the opposite trade to the
weak-coupling route surveyed in the introduction.
\end{example}

\begin{example}[A non-perturbative full $U$-law model: the mean-field dealer market]\label{ex:dealer}
Read the control problem of Example~\ref{ex:concrete}(i) as a dealer managing a one-factor
marked-to-market exposure $X$: requests for quotes arrive as the marked points of $N$, the
mark $e$ recording side and size, and a fill of mark $e$ moves the exposure through the
jump channel. The Brownian term is the dealer's residual hedging risk; marking to market
removes the common reference price from the strategic problem. The adjoint component $U_t$ is then the dealer's \emph{valuation curve}:
$U_t(e)$ is the marginal continuation value of a fill of mark $e$, so $U_t$ is an element
of $\Lnu$ indexed by side and size, and a population of dealers carries a law on curves,
$\mu^U\in\mathcal P_2(\Lnu)$. Quotes are affine functions of this curve, so the
population's quote-curve distribution is an affine image of $\mu^U$, and an anonymous
routing mechanism that compares each dealer's curve with the population's tilts the filled
flow by the \emph{crowding correction}
\[
\mathcal A_\lambda(u):=\int_{\Lnu}\nabla V(u-v)\,\lambda(\diff v),\qquad
V(w):=\sqrt{1+\Lnorm{w}^2}-1,\qquad \lambda\in\mathcal P_2(\Lnu).
\]
The potential $V$ is even and convex, with
\[
\nabla V(w)=\frac{w}{\sqrt{1+\Lnorm{w}^2}},\qquad \Lnorm{\nabla V(w)}\le1,\qquad
\mathrm{Lip}(\nabla V)\le1,
\]
the last because the second derivative of $V$ has operator norm at most one.
Fix $\varepsilon\ge0$ and tilt the jump response of Example~\ref{ex:concrete}(i) by the
crowding correction:
\[
b=-\beta_1 y,\qquad \sigma=\sigma_c-\beta_1 z,\qquad
h=h_c-\beta_1 u-\varepsilon\,\mathcal A_{\mu^U}(u),\qquad
f=\beta_2 x,\qquad g(x,\rho)=x ,
\]
with $\mu^U$ the $U$-marginal of the measure argument. At $\varepsilon=0$ this is
Example~\ref{ex:concrete}(i) verbatim.

\smallskip
\emph{Assumptions~\ref{ass:integ}--\ref{ass:G}.} Since $\nabla V(0)=0$, the origin value
$h(\cdot,0,\delta_0)=h_c$ is unchanged, so Assumption~\ref{ass:integ} holds as in
Example~\ref{ex:concrete}(i). For $u,u'\in\Lnu$ and
$\lambda,\lambda'\in\mathcal P_2(\Lnu)$, coupling $(v,v')$ optimally for
$(\lambda,\lambda')$ and using that $\nabla V$ is $1$-Lipschitz,
\[
\Lnorm{\mathcal A_\lambda(u)-\mathcal A_{\lambda'}(u')}
\ \le\ \Lnorm{u-u'}+W_2(\lambda,\lambda') ,
\]
and $W_2(\mu^U,\mu'^U)\le W_2(\mu,\mu')$ by \eqref{eq:margproj}, so
Assumption~\ref{ass:lip} holds with $L=\max\{\beta_1+\varepsilon,\beta_2,1\}$: the
interaction strength enlarges the Lipschitz constant only. Joint Lipschitz continuity of
$(u,\lambda)\mapsto\mathcal A_\lambda(u)$ gives the Borel
measurability~\eqref{eq:hborel}. Assumption~\ref{ass:G} holds as in
Example~\ref{ex:concrete}(i).

\smallskip
\emph{Monotonicity at arbitrary $\varepsilon$.} Let
$(\Theta,\mu),(\Theta',\mu')\in\mathcal D_t$; taking $U$-marginals of $\mu=\Law(\Theta)$
and $\mu'=\Law(\Theta')$, $\mu^U=\Law(U)$ and $\mu'^U=\Law(U')$. Let
$(\widehat U,\widehat U')$ be an independent copy of $(U,U')$: then $\widehat U$ has law
$\mu^U$ and is independent of $(U,U')$, so
\[
\mathcal A_{\mu^U}(U)=\int_{\Lnu}\nabla V(U-v)\,\mu^U(\diff v)
=\E\bigl[\nabla V(U-\widehat U)\,\big|\,U\bigr],
\]
and likewise for $(U',\mu'^U)$. Since
exchanging $(U,U')$ with $(\widehat U,\widehat U')$ preserves the joint law while
$\nabla V$ is odd,
\[
\begin{aligned}
\E\inner{\mathcal A_{\mu^U}(U)-\mathcal A_{\mu'^U}(U')}{U-U'}
&\ =\ \E\inner{\nabla V(U-\widehat U)-\nabla V(U'-\widehat U')}{U-U'}\\
&\ =\ -\,\E\inner{\nabla V(U-\widehat U)-\nabla V(U'-\widehat U')}{\widehat U-\widehat U'} .
\end{aligned}
\]
Averaging the two lines,
\[
\begin{aligned}
&\E\inner{\mathcal A_{\mu^U}(U)-\mathcal A_{\mu'^U}(U')}{U-U'}\\
&\qquad=\tfrac12\,\E\inner{\nabla V(U-\widehat U)-\nabla V(U'-\widehat U')}
{(U-\widehat U)-(U'-\widehat U')}\ \ge\ 0,
\end{aligned}
\]
the final inequality because $\nabla V$, the gradient of the convex $V$, is monotone.
The $\varepsilon$-term therefore improves the interior pairing,
\[
\E\int_E\inner{\delta h(e)}{\delta U(e)}\,\nu(\diff e)\ \le\ -\beta_1\,\E\Lnorm{\delta U}^2 ,
\]
and \eqref{eq:mono} holds with the constants $\alpha=1,\beta_1,\beta_2$ of
Example~\ref{ex:concrete}(i) for \emph{every} $\varepsilon\ge0$ --- no smallness
condition and no margin. Theorem~\ref{thm:cwp} applies directly.

\smallskip
\emph{Beyond the mean.} The crowding correction is not a function of $\E[U]$: for a
unit $w\in\Lnu$, the two-point laws $\lambda_1=\tfrac12(\delta_{-w}+\delta_{w})$ and
$\lambda_2=\tfrac12(\delta_{-2w}+\delta_{2w})$ share their mean, yet
\[
\mathcal A_{\lambda_1}(w)=\tfrac{1}{\sqrt5}\,w,\qquad
\mathcal A_{\lambda_2}(w)=\tfrac12\Bigl(\tfrac{3}{\sqrt{10}}-\tfrac1{\sqrt2}\Bigr)w .
\]

\smallskip
\emph{Marks and activity.} Take $E=\{-1,+1\}\times(0,1]$, writing $e=(\varsigma,q)$ for
side $\varsigma$ and size $q$, with
$\nu(\diff\varsigma,\diff q)=c_\varsigma\,q^{-1-\varrho}\,\diff q$ for constants
$c_\varsigma>0$ and $0<\varrho<2$: then $\nu(E)=\infty$ while
$\int_E q^2\,\nu(\diff\varsigma,\diff q)<\infty$, so $h_c(\varsigma,q):=\varsigma q$ lies
in $\Lnu$ --- infinitely many small tickets with square-integrable aggregate exposure. In
this model the infinite-activity regime is the market's small-ticket limit.

The finite-player inspiration is the dealer market of Cont and
Xiong~\cite{ContXiong}, in which dealers' execution probabilities depend on their own
and competing quotes; the present example is a stylised mean-field model inspired by that
setting.
\end{example}

\begin{remark}[Scope of the measure argument]\label{rem:freelaw}
The proofs consume the tuple law through a single mechanism: the $W_2$ term of
Assumption~\ref{ass:lip}, dominated at every use by the synchronous coupling
bound~\eqref{eq:coupling}. The right-hand side of \eqref{eq:coupling} carries
$\E\Lnorm{\delta U}^2$ whether or not the measure argument reads the $U$-marginal,
because the coefficients depend on $u$ pointwise --- a dependence
Remark~\ref{rem:beta1-dich} makes unavoidable whenever $\beta_1>0$. Nothing is therefore gained by restricting the
coefficients to the $(X,Y,Z)$-marginal law: with the coupling bound taken in
$\mathcal H$, the estimates close at the full tuple law, and Example~\ref{ex:dealer}
exhibits that scope arising from a routing mechanism rather than by construction.
\end{remark}

\begin{remark}[Perturbative and non-perturbative law dependence]\label{rem:regimes}
Example~\ref{ex:fulllaw} shows that arbitrary nonlinear full-tuple law dependence can be
added when it is smaller than the dissipativity margin. Example~\ref{ex:dealer} shows
that a genuinely nonlinear $U$-law interaction can have arbitrary strength when its
geometry is itself monotone.
\end{remark}

 \section{Conclusion}\label{sec:conclusion}

Theorems~\ref{thm:stab}, \ref{thm:uniq} and~\ref{thm:cwp} establish global well-posedness, with a quantitative
stability estimate, for fully coupled MV-FBSDEJ with $W_2$-dependence on the law of the whole
tuple, including the $\Lnu$-valued jump integrand, at arbitrary $\sigma$-finite jump activity.
By Corollary~\ref{cor:robust}, the well-posedness is robust, at $m=n$, under full-tuple-law
perturbations below the dissipativity margin. The Lipschitz and monotonicity assumptions are imposed only along the diagonal inputs
$(\Theta_t,\Law(\Theta_t))\in\mathcal D_t$:
Proposition~\ref{prop:strict} shows that expected diagonal $G$-monotonicity is strictly
weaker than its pointwise form (Remark~\ref{rem:diag-mono}). The monotonicity of
Assumption~\ref{ass:mono} is the principal structural assumption and is still coercive; in the $m<n$ regime it
forces strictly $U$-dissipative jump feedback (Remark~\ref{rem:beta1-dich}). Relaxing that coercivity, so that additive jump coefficients
become admissible when $m<n$, and treating domination-monotonicity
conditions~\cite{TianYu,WuHao} or conditional laws under common noise, are the natural next
questions.

 \clearpage
\appendix
\section{Predictable representatives and the law flow}\label{app:lawflow}

\begin{lemma}[Predictable representatives and the full-tuple law flow]\label{lem:lawflow}
Let $X,Y$ be c\`adl\`ag adapted and let $Z,U$ be fixed predictable versions, with
\[
\E\int_0^T\Bigl(|X_t|^2+|Y_t|^2+|Z_t|_F^2+\int_E|U_t(e)|^2\,\nu(\diff e)\Bigr)\diff t
<\infty .
\]
Then the a.e.-defined maps
$t\mapsto\Law(\Theta_t^-)$ and $t\mapsto\Law(\Theta_t)$ admit Borel representatives,
unique up to Lebesgue-null sets of times,
\[
\mu^-,\,\mu:[0,T]\to(\mathcal P_2(\mathcal H),W_2),
\qquad
\mu_t^-=\mu_t\quad\text{for a.e.\ }t .
\]

Coefficient evaluations along these representatives are invariant, up to
$\diff t\otimes\diff\Prob$-null sets, under the choice of predictable representatives of
$(Z,U)$ and of the Borel representatives.
\end{lemma}
\begin{proof}
\emph{Step 1 (the field).} The left limits $X_-,Y_-$ are predictable and agree with $X,Y$
$\diff t\otimes\diff\Prob$-a.e., as c\`adl\`ag paths jump at countably many times, so
$\int_0^T\mathbf 1_{\{X_{t-}\neq X_t\}}\,\diff t=0$ almost surely; the hypothesis on
$X,Y$ thus passes to them, and $\mathcal H$ is separable because $\Lnu$ is. The
$Z$-coordinate is predictable and square-integrable by assumption. For the
$U$-coordinate, the integrability hypothesis and the Fubini
isometry~\eqref{eq:fubiniiso} give
\[
U\in L^2\bigl([0,T]\times\Omega\times E\bigr)\cong L^2\bigl([0,T]\times\Omega;\Lnu\bigr),
\]
so the induced map $(t,\omega)\mapsto U_t(\omega,\cdot)$ admits a predictable, strongly
measurable $\Lnu$-valued representative. $\Theta^-$ is therefore a
predictable, hence jointly measurable, $\mathcal H$-valued field with
\[
\E\int_0^T\|\Theta_t^-\|_{\mathcal H}^2\,\diff t<\infty,
\qquad\text{that is,}\qquad
\Theta^-\in L^2\bigl([0,T]\times\Omega;\mathcal H\bigr).
\]
The Fubini isometry of~\cite[Proposition~1.2.24]{HNVW},
\[
L^2\bigl([0,T]\times\Omega;\mathcal H\bigr)\ \cong\ L^2\bigl(0,T;L^2(\Omega;\mathcal H)\bigr),
\]
supplies a strongly measurable
$L^2(\Omega;\mathcal H)$-valued representative of $t\mapsto\Theta_t^-$; by the Pettis
measurability theorem~\cite[Theorem~1.1.6]{HNVW}, that map takes values in a separable
closed subspace of $L^2(\Omega;\mathcal H)$ off a Lebesgue-null set of times, and
therefore admits a Borel version, agreeing with it for a.e.\ $t$.

\emph{Step 2 (Borel flow).} By the synchronous coupling bound~\eqref{eq:coupling}, the
law map $\Theta\mapsto\Law(\Theta)$ is $1$-Lipschitz from $L^2(\Omega;\mathcal H)$ into
$(\mathcal P_2(\mathcal H),W_2)$, which inherits Polishness from
$\mathcal H$~\cite[Ch.~6]{Villani2009}. Composing it with the Borel version of Step~1 gives a
Borel map $\mu^-:[0,T]\to\mathcal P_2(\mathcal H)$ with $\mu_t^-=\Law(\Theta_t^-)$
for a.e.\ $t$; any two such representatives agree a.e. The same argument applied to
the progressively measurable field $\Theta$ yields $\mu$; as $\Theta$ and $\Theta^-$
agree $\diff t\otimes\diff\Prob$-a.e., $\mu_t=\mu_t^-$ for a.e.\ $t$.

\emph{Step 3 (independence of representatives).} Let $(Z',U')$ be another pair of
predictable versions, so $Z=Z'$ $\diff t\otimes\diff\Prob$-a.e.\ and $U=U'$
$\diff t\otimes\diff\Prob\otimes\nu$-a.e.; by Fubini the latter gives $U=U'$ in $\Lnu$
for $\diff t\otimes\diff\Prob$-a.e.\ $(t,\omega)$. The two fields therefore agree in
$L^2(\Omega;\mathcal H)$ for a.e.\ $t$, whence $\mu_t^-=(\mu_t^-)'$ for a.e.\ $t$. The
coefficients are therefore evaluated at arguments coinciding
$\diff t\otimes\diff\Prob$-a.e.\ (in $\Lnu$ for the jump coefficient); changing the
Borel representative likewise alters $\mu^-$ only on a Lebesgue-null set of times, hence
the arguments only on a $\diff t\otimes\diff\Prob$-null set. In either case the
coefficient processes represent the same $L^2$ classes over their respective product
measures, so every Lebesgue and stochastic integral built from them, hence every
identity and estimate, is unchanged.
\end{proof}

 \end{document}